\documentclass[11pt]{amsart}
\usepackage[margin=3cm]{geometry}               
\usepackage[usenames,dvipsnames,svgnames,table]{xcolor}    
\usepackage[all]{xypic}
\usepackage{graphicx}
\usepackage{tikz}
\usetikzlibrary{decorations.pathreplacing}
\usepackage{lineno, hyperref}
\usepackage{amssymb,amsfonts, amsmath, amsthm,bbm}
\usepackage{epstopdf,enumitem}
\usepackage{ulem}
\usepackage{cancel}
\usepackage{thmtools}
\newtheorem*{theorem*}{Theorem}

\declaretheorem[style=definition,qed=$\blacksquare$,numberwithin=section]{definition}
\declaretheorem[style=definition,qed=,sibling=definition]{remark}
\declaretheorem[style=definition,qed=$\blacksquare$, sibling = definition]{lemma-definition}

\declaretheorem[style=theorem, sibling = definition]{theorem}
\declaretheorem[style=theorem, sibling = definition]{lemma}
\declaretheorem[style=theorem, sibling = definition]{proposition}
\declaretheorem[style=theorem, sibling = definition]{claim}

\declaretheorem[style=definition, sibling = definition]{example}
\declaretheorem[style=definition, sibling = definition]{conjecture}

\newtheorem{innercustomthm}{Theorem}
\newenvironment{customthm}[1]
  {\renewcommand\theinnercustomthm{#1}\innercustomthm}
  {\endinnercustomthm}

\newcommand{\Q}{\mathbb{Q}}
\newcommand{\QQ}{\Q}

\newcommand{\Z}{\mathbb{Z}}
\newcommand{\ZZ}{\Z}
\newcommand{\N}{\mathbb{N}}
\newcommand{\NN}{\mathbb{N}}

\newcommand{\C}{\mathbb{C}}

\newcommand{\Lm}{\mathcal{L}}
\newcommand{\cM}{\mathcal{M}}
\newcommand{\cO}{\mathcal{O}}

\newcommand{\ord}{\mathrm{ord}}

\newcommand{\Lring}{\Lm_{\text{ring}}}

\newcommand{\ac}{\text{ac}\,}

\def\cL{{\mathcal L}}
\def\cC{{\mathcal C}}
\def\CC{{\mathbb C}}
\def\BB{{\mathbb B}}
\def\cB{{\mathcal B}}

\newcommand{\pball}{\mathbb{B}_{1}}

\title{Exponential-constructible functions in $P$-minimal structures}

\author{Saskia Chambille}

\author{Pablo Cubides Kovacsics }

\author{Eva Leenknegt}

\address{Saskia Chambille, KU Leuven, Department of Mathematics, Celestijnenlaan 200B, 3001 Heverlee, Belgium }
\email{saskia.chambille@kuleuven.be}
\address{Pablo Cubides Kovacsics, Universit\'e de Caen, Laboratoire de math\'ematiques Nicolas Oresme, CNRS UMR 6139,
14032 Caen cedex, France }
\email{pablo.cubides@unicaen.fr}
\address{Eva Leenknegt, KU Leuven, Department of Mathematics, Celestijnenlaan 200B, 3001 Heverlee, Belgium }
\email{eva.leenknegt@kuleuven.be}

\begin{document}


%
%
%


\maketitle

\begin{abstract}
Exponential-constructible functions are an extension of the class of constructible functions. This extension was formulated by Cluckers-Loeser in the context of semi-algebraic and sub-analytic structures, when they studied stability under integration.

In this paper we will present a natural refinement of their definition that allows for stability results to hold within the wider class of P-minimal structures. One of the main technical improvements is that we remove the requirement of definable Skolem functions from the proofs. As a result, we obtain stability in particular for all intermediate structures between the semi-algebraic and the sub-analytic languages.


\end{abstract}


\normalem

\section{Introduction}
 
Stability under integration is a problem that has been studied in many different contexts. In this paper we will revisit some existing results, and generalize them to the wider context of $P$-minimality. Part of the difficulty lies in considering $P$-minimal expansions which do not have Skolem functions. As we shall explain later, in situations where such functions do not exist, many of the classical strategies fail. New ideas are used in order to avoid this assumption.
\\\\
In the first part of this introduction we will briefly recall the existing stability results for constructible and exponential-constructible functions. In the second part we will introduce and motivate our refinement of the class of exponential-constructible functions and in the third part we will state our main results.

\subsection{The algebra $\cC_\textnormal{exp}$ of exponential-constructible functions}
Let $K$ be a $p$-adically closed field, i.e. a field elementarily equivalent to a finite extension of $\Q_p$. We use the notation $\ord:K\to\Gamma_K\cup\{\infty\}$ for the valuation map, $\Gamma_K$ for the value group, $\cO_K$ for the valuation ring of $K$, $\mathcal{M}_K$ for the maximal ideal, $q_K$ for the number of elements of the residue field $k_K$ and $\pi_K$ for a uniformizing element. Let us first recall the definition of $P$-minimality.
\begin{definition}[Haskell, Macpherson \cite{has-mac-97}]
A structure $(K,\Lm)$ is called $P$-minimal if $K$ is a $p$-adically closed field, $\Lm \supseteq \Lring$, and for every structure $(K', \Lm)$ elementarily equivalent to $(K, \Lm)$, one has that the $\Lm$-definable subsets of $K'$ coincide with the $\Lring$-definable subsets of $K'$. 
\end{definition}
It is natural to extend this notion to two-sorted structures $(K, \Gamma_K; \Lm_2)$, by extending the language $\Lm$ to a two-sorted language $\Lm_2$. For the value group sort $\Gamma_K$, the language $\Lm_2$ is the Presburger language for ordered abelian groups. We also add the valuation map $\ord$ as a connective between both sorts. The minimality notion remains the same, as no further minimality requirements are put on the value group sort. For a more in-depth discussion of $2$-sorted $P$-minimality, we refer to \cite[Section 2]{cubi-leen}. We will write definable rather than $\Lm_2$-definable when the language used is clear from the context.
\\\\
In this paper we will be considering integrals of certain classes of functions. As integration is not necessarily well-defined for $p$-adically closed fields in general, our theorems concerning integration will only be phrased for $p$-adic fields, i.e., finite extensions of $\Q_p$, where $\Gamma_K = \Z$.
Integration will be with respect to the Haar measure on $K$ (normalised such that $\cO_K$ has measure 1), and the counting measure on $\ZZ$. Given definable sets $S$ and $Y$ (which may contain both $K$-sorted and $\ZZ$-sorted variables), let $X\subseteq S\times Y$ be definable. We will use the notation $\pi_S(X)$ for the projection of $X$ onto $S$. For functions $f:X \to \C$, the locus of integrability of $f$ with respect to $Y$ is defined as
\[
\text{Int}(f,Y)=\{s\in S: f(s,\cdot) \text{ is measurable and integrable over $X_s$} \}.
\]\\
Let us now introduce the classes of functions we are interested in. Consider an additive character $\psi: K \to \CC^\times$, such that $\psi_{|\mathcal{M}_K} = 1$ and $\psi_{|\cO_K} \neq 1$.
An example of such a function on $\QQ_p$ is 
\[
\psi(x)=\exp\left(\frac{2\pi ix'}{p}\right),
\]
where $x' \in \ZZ[\frac{1}{p}] \cap (x + \mathcal{M}_K)$.
\\\\The algebras of ``constructible'' and ``exponential-constructible'' functions were originally introduced by Denef \cite{denef-85} and Cluckers-Loeser \cite{clu-loe-08} respectively. The following definition is a rephrasing using the terminology of two-sorted structures.



\begin{definition}\label{def:constructible}
Let $(K, \Z; \Lm_2)$ be a $P$-minimal structure and $X$ a definable set. 
\begin{itemize}
\item[(i)] The algebra $\cC(X)$ of \emph{$\cL_2$-constructible functions on $X$} is the $\QQ$-algebra generated by constant functions and functions of the forms 
\[\alpha:X\to \ZZ \qquad \text{and} \qquad
X\to \QQ:x\mapsto q_K^{\beta(x)},\]
where $\alpha$ and $\beta$ are definable and $\ZZ$-valued.
\item[(ii)]The algebra $\cC_{\textnormal{exp},\psi}(X)$ of \emph{$\cL_2$-exponential-constructible functions on $X$} is the $\QQ$-algebra generated by functions in $\cC(X)$ and functions of the form $\psi\circ f$ where $f:X\to K$ is definable. \qedhere
\end{itemize}
\end{definition}
\noindent We will write $\cC_\textnormal{exp}(X)$ rather than $\cC_{\textnormal{exp},\psi}(X)$ whenever $\psi$ is clear from the context.  
\\\\
Stability for constructible functions was first studied Denef \cite{denef-85} for semi-algebraic structures (i.e. structures over $\Lring$). Later his results were generalized by Cluckers \cite{clu-2003} to sub-analytic structures (i.e. structures over $\Lm_{\text{an}}$ - see e.g. \cite{clu-2003} for a definition).  These results, along with similar results by Cluckers-Loeser \cite{clu-loe-08} for exponential-constructible functions, give a type of stability that can be defined as follows:

\begin{definition}
Let $(K,\Z; \Lm_2)$ be a two-sorted structure. We say that a class $\mathcal{H}$ of $\mathbb{C}$-valued functions on this structure is \emph{base-stable under integration} if, for every definable set $X\subseteq S\times Y$ and $f:X \to \C$ with $f\in \mathcal{H}$ and $\textnormal{Int}(f,Y)=S$, there exists $g: S \to \C$ such that $g\in \mathcal{H}$ and for all $s\in S$,
\[
g(s) = \int_{X_s} f(s,x) |dx|.  
 \] 
We say that $\mathcal{H}$ is \emph{base-stable under integration over $K$-variables} if the previous condition holds for all definable sets $X\subseteq S\times Y$ where $Y\subseteq K^m$ for some $m\geqslant 1$.
\end{definition}
Denef and Cluckers first proved the following.
\begin{theorem}[Denef, Cluckers] \label{thm:d-c} If $\cL$ is either $\cL_\textnormal{ring}$ or $\cL_\textnormal{an}$, then  the algebras of constructible functions are base-stable under integration over $K$-variables.
\end{theorem}

Cell decomposition theorems in the style of Denef \cite{denef-86} are the main tool used in the proofs of Theorem \ref{thm:d-c}. Note however that, by results of Mourgues \cite{mou-09} and Darni\`ere-Halupzcok \cite{Dar-Hal-2017}, a $P$-minimal field satisfies such a \emph{classical} cell decomposition theorem if and only if it admits definable Skolem functions as well. This means that to generalize such stability results to the full class of $P$-minimal structures, a somewhat alternative approach is needed. Moreover, the second author and Nguyen recently provided an example showing that $P$-minimal structures that do not admit definable Skolem functions do indeed exist. In order to obtain the generalization stated in Theorem \ref{thrm:Cubides-Leenknegt} below, the second and third author  \cite{cubi-leen} used a weaker cell decomposition theorem valid in all $P$-minimal structures.
\begin{theorem}\label{thrm:Cubides-Leenknegt}
Let $(K, \ZZ; \cL_2)$ be $P$-minimal. Then the algebras of constructible functions are base-stable under integration.  
\end{theorem}
The aim of this paper is to make similar generalizations for the case of exponential-constructible functions, valid for all $P$-minimal expansions of $p$-adic fields. The clustered cell decomposition proven by the authors in \cite{cham-cubi-leen}--an enhancement of the weak cell decomposition theorem proven in \cite{cubi-leen}-- will play a crucial role in our proofs. 
\\\\
From now on, we will refer to the \emph{Skolem setting} when working over $P$-minimal fields that admit definable Skolem functions (and hence satisfy a classical cell decomposition theorem by \cite{mou-09, Dar-Hal-2017}). The term \emph{non-Skolem setting} refers to situations where we explicitly assume we are working over   $P$-minimal fields that do not admit such definable Skolem functions. When we make no assumptions either way, we will refer to the 
\emph{$P$-minimal setting}. Also, when referring to \emph{general $P$-minimal structures}, this means that we are considering not only $P$-minimal expansions of $p$-adic fields but rather $P$-minimal expansions of arbitrary $p$-adically closed fields.

Before we state our own results, let us first take a moment to look at prior results. The original result by Cluckers-Loeser required a further assumption on the form of definable functions:


\begin{theorem}[Cluckers-Loeser] \label{thm:cluckers-loeser} Let $\cL$ be either $\cL_\textnormal{ring}$ or $\cL_\textnormal{an}$, $X\subseteq S \times K^n$ a definable set and $f\in \cC_\textnormal{exp}(X)$ with
 \begin{equation} \label{eq:cluckers-loeser}f(s,x) = \sum_{i=1}^m h_{i}(s,x)\psi(f_i(s,x)),\end{equation}
 where $h_i \in \mathcal{C}(X)$ with $\text{Int}(h_i,K^n)=S$ and $f_i$ is a definable function. Then there exists $g \in \mathcal{C}_{\textnormal{exp}}(X)$ such that for all $s \in S$, 
 \[g(s) = \int_{X_s} f(s,x)|dx|.\] 
\end{theorem}
In a subsequent paper, Cluckers, Gordon and Halupczok \cite{Clu-Gor-Hal-14} managed to remove the condition \eqref{eq:cluckers-loeser} on the form of $f$, thereby showing that for $\Lring$ and $\Lm_{\text{an}}$, algebras of exponential-constructible functions are always base-stable under integration over $K$-variables. (Or to be more precise, their result was even stronger as they also managed to remove the condition on the locus of integrability.)
\\\\
In order to obtain results valid in the non-Skolem setting,
we will need to slightly adapt the definition of exponential-constructible functions.

\subsection{The algebra $\cC_{\textnormal{exp}}^*$ of exponential*-constructible functions}
Before providing the definition of $\cC_{\textnormal{exp}}^*$, let us informally explain why the algebra $\cC_{\textnormal{exp}}$ will need to be adapted to suit our purposes. We will need the following notation and definitions. 

The ball with (valuative) radius $\gamma$ and center $a$ will be denoted as
\[B_\gamma(a):= \{ x \in K \mid \ord(x-a) \geqslant \gamma\}.\]
The set consisting of all balls with a given radius $\gamma$ will be denoted as 
\[\mathbb{B}_{\gamma}:= \{B_\gamma(a) \mid a \in K\}.\]
\begin{definition}
Let $X \subseteq K$. Let $B \subseteq X$ be a ball such that for all balls $B'\subseteq X$, one has that \(B \subseteq B' \Rightarrow B = B'.\) Then we call
$B$ a \emph{maximal ball} of $X$. This property will be denoted as $B \sqsubseteq X$.
\end{definition}

\begin{definition}\label{def:multiball} Let $k \in \N \setminus\{0\}$.
 A set $A \subseteq S \times K$ is called a \emph{multi-ball of order k} if every fiber $A_s$ is a union of $k$ disjoint balls of the same radius. \item A multi-ball of order $k$ is said to be \emph{on $\mathbb{B}_{\gamma}$} if, for each fiber $A_s$, the $k$ balls $B$ contained in $A_s$ all satisfy $B \sqsubseteq A_s$ and $B \in \mathbb{B}_{\gamma}$.
\end{definition}
\noindent We may not always explicitly mention the order of a multi-ball, but even in such cases, the order will always be assumed finite.
\\\\
In the next example we will compute the integral of a very simple exponential-constructible function to illustrate the type of difficulties one may encounter when  definable Skolem functions are not available.

\begin{example} \label{ex: motivation exp*}
Let $k \in \NN \setminus\{0\}$ and let $A \subseteq S \times K$ be a definable multi-ball of order $k$ on $\pball$. Consider the exponential-constructible function $f\colon  A \to \CC : (s,x) \mapsto \psi(x)$. Using the fact that $\psi$ is constant on balls in $\pball$, we get that
\[
\int_{A_s} f(s,x) |dx| = \int_{A_s} \psi(x) |dx| = \sum_{B \sqsubseteq A_s} \psi(B) \cdot \text{Vol}(B) = q_K^{-1} \cdot \sum_{B \sqsubseteq A_s} \psi(B).
\]
If a structure admits definable Skolem functions,  there exist definable sections of $A$, say $f_1, \ldots, f_k\colon  S \to K$, such that
\begin{equation}\label{eq:exp*=exp}
\sum_{B \sqsubseteq A_s} \psi(B) = \sum_{j=1}^k \psi(f_j(s)).
\end{equation}
Hence, the function $s \mapsto \int_{A_s} f(s,x) |dx|$ will be an element of $\cC_\textnormal{exp}(S)$. Note that if a structure does not admit Skolem functions, one cannot always make this type of substitution.
\end{example}

The next example shows that the character of a definable function can always be written as a character sum over a multi-ball:
\begin{example}\label{ex: motivation exp*2}
Let $f\colon  X \to K$ be a definable function. For each $x \in X$, let $A_x$ be the ball $A_x := f(x) + \cM_K$. Then the set $A:=\{(x,y) \in X \times K \mid y \in A_x\}$ is a definable multi-ball of order 1 on $\mathbb{B}_1$. Moreover, for all $x \in X$ one has that 
 \[\psi(f(x)) = \sum_{B \sqsubseteq A_x} \psi(B).\]
\end{example}
%
%

To summarize, Example \ref{ex: motivation exp*} shows that the algebras of exponential constructible functions will need to be extended (if one wants base-stability under integration 
in the $P$-minimal setting), 
 and Example \ref{ex: motivation exp*2} indicates that working with functions of the form $x \mapsto \sum_{B \sqsubseteq A_x} \psi(B)$, rather than $x \mapsto \psi(f(x))$,  yields a very natural generalization  to a wider setting. This motivates us to propose the following definition: 
\begin{definition}\label{def:exp*functions}
Let $(K, \Z, \Lm_2)$ be a $P$-minimal structure and $X$ be a definable set. The algebra $\cC_{\textnormal{exp},\psi}^*(X)$ of \emph{$\cL_2$-exponential*-constructible functions on $X$} is the $\QQ$-algebra generated by functions in $\cC(X)$ and functions of the form $x \mapsto \sum_{B\sqsubseteq A_x} \psi(B)$, where $A\subseteq X\times K$ is a definable multi-ball on $\pball$.
\end{definition}


\begin{remark}\label{rem: motivation exp*}
Note that  in the Skolem setting, the identity \eqref{eq:exp*=exp} holds, hence any exponential*-constructible function is also exponential-constructible for such structures.
\end{remark}

\subsection{Overview of main results}
We are now ready to state the main theorems of this paper. We will always work in a $P$-minimal structure $(K, \Z,\Lm_2)$ unless explicitly stated otherwise. 


\begin{customthm}{A}\label{thm:Zvar} Let $X\subseteq S\times \ZZ^n$ be a definable set. Let $f\in \cC_{\exp}^*(X)$ be such that $\textnormal{Int}(f,\ZZ^n) = S$. Then there exists $g\in \cC_{\exp}^*(S)$ such that for all $s \in S$,
\[
g(s) = \int_{X_s} f(s,\gamma)  |d\gamma|.  
 \]
\end{customthm}

 Let $X\subseteq S\times K^n$ be a definable set and $f\in \cC^*_\textnormal{exp}(X)$ be such that $\textnormal{Int}(f,Y)=S$. The main obstruction to deriving the stability under integration for $f$ is related to integrability conditions on some of the functions used to define $f$. To formalize this let us introduce some terminology. 

\begin{definition}\label{def:normalform}
For $n \geqslant 1$, let $X \subseteq S \times K^n$ be a definable set. We say that a function $f\in \cC^*_\text{exp}(X)$ can be written in \emph{$n$-normal form}, if there exists a definable partition $X = \cup_{w\in W} X_w$ (where $W$ is a finite index set), such that on each $X_w$ the following holds:
\begin{enumerate}
\item There exist $m\geqslant 1$ and functions $f_1, \ldots, f_{m} \in  \cC^*_\text{exp}(X_w)$ such that $f_{|X_w}=\sum_{i=1}^{m} f_i$,
\item Each function $f_i$ can be further expanded as \[f_i(s,x)=h_i(s,x) \sum_{B\sqsubseteq A_{s,x}^i} \psi(B), \qquad \text{where}\] 
\begin{enumerate}
\item $h_i\in\cC(X_w)$ and $\text{Int}(h_i,K^n)=\pi_S(X_w)$,
\item $A^i$ is a definable multi-ball over $X_w$ of order $k_i$ on $\pball$,
\item 
for each $s \in \pi_S(X_w)$, $x \mapsto \sum_{B\sqsubseteq A_{s,x}^i} \psi(B)$ is a measurable function in $x$.
\qedhere
\end{enumerate}
\end{enumerate}
\end{definition}

\begin{customthm}{B}\label{thm:Qpvar}  Let $X \subseteq S \times K^n$ be a definable set and $f\in \cC^*_\text{exp}(X)$. If $f$ can be written in $n$-normal form, then there exists $g\in \cC_{\exp}^*(S)$ such that, for all $s \in S$,
 \[
g(s) = \int_{X_s} f(s,x) |dx|.
 \]
\end{customthm}

Note that our notion of $n$-normal form is similar in nature to the assumptions on the form of $f$ made in Theorem \ref{thm:cluckers-loeser}. 
Now consider the following conjecture.


\begin{conjecture}\label{conj:normalform}
Let $X\subseteq S\times K$ be a definable set and $f\in \cC^*_\textnormal{exp}(X)$ such that $\textnormal{Int}(f,K)=S$. Then $f$ can be written in $1$-normal form.  
\end{conjecture}

Under the above conjecture, Theorems \ref{thm:Zvar} and \ref{thm:Qpvar} imply that the algebras of exponential*-constructible functions are base-stable under integration. As the proof is rather short, we will include it here in the introduction for the reader's convenience. 


\begin{customthm}{C}\label{mainthrm}
Suppose the conjecture holds. Then the algebras of exponential*-constructible functions are base-stable under integration.
\end{customthm}

\begin{proof}
Let $X\subseteq S\times Y$ be a definable set and $f\in \cC^*_\textnormal{exp}(X)$ be such that $\textnormal{Int}(f,Y)=S$. By Fubini, it suffices to consider the cases where $Y=K$ or $Y=\ZZ$. The case $Y = \ZZ$ follows from Theorem \ref{thm:Zvar}. For the case $Y = K$, the conjecture implies that there exists a finite partition $X = \cup_{w\in W} X_w$ such that for each $X_w$, we can write $f_{|X_w}(s,x)=\sum_{i=1}^{m} f_i(s,x)$, where the functions $f_i$ satisfy condition $(2)$ of Definition \ref{def:normalform}. This implies that for all $s\in S_w :=\pi_S(X_w)$, each $f_i(s,\cdot)$ is integrable over $X_s$, hence
\[
 \int_{(X_w)_s} f(s,x) |dx|=  \int_{(X_w)_s} \sum_{i=1}^{m} f_i(s,x)|dx| = \sum_{i=1}^{m} \int_{(X_w)_s} f_i(s,x)|dx|. 
\]
By Theorem \ref{thm:Qpvar} there exists, for each $i\in\{1,\ldots,m\}$, a function $g_i\in \cC^*_\textnormal{exp}(S_w)$  such that $g_i(s) = \int_{(X_w)_s} f_i(s,x) |dx|$. Remark \ref{rmk:extendingfunctions} below allows us to extend the $g_i$ to functions in $\cC^*_{\exp}(S)$. Putting $g_w(s):=\sum_{i=1}^{m} g_i(s)$ and summing over all $w \in W$ completes the proof.
\end{proof}
\begin{remark}\label{rmk:extendingfunctions}
If $U \subseteq X$ are definable sets, then for any $f\in \cC^*_{\exp}(U)$, there exists $f_X \in \cC^*_{\exp}(X)$, such that
\[
f_X(x) = \begin{cases} f(x) &\textnormal{if } x \in U; \\ 0 &\textnormal{if not.}\end{cases}
\]
We will often abuse notation and simply write $f$ rather than $f_X$. This trick will be used when partitioning the domain $X$ of an exponential*-constructible function, as it allows us to extend functions on one of the sets in the partition to functions on $X$.  We may not always explicitly mentioned this. 
\end{remark}

It is worth noting that in \cite{Clu-Gor-Hal-14}, removing the assumption \eqref{eq:cluckers-loeser} on the form of $f$ in the case of $\Lring$ or $\Lm_{\text{an}}$ (by proving a variation on the above conjecture) required a proof that made use of the Jacobian Property (see \cite[Proposition 3.3.5]{Clu-Gor-Hal-14}). At the time of writing, it is still an open question as to whether (a version of) that property holds in the $P$-minimal setting or even in the Skolem setting. Currently only a local version is known (see \cite{kui-lee-13}). Note that Conjecture \ref{conj:normalform} is open in the Skolem setting as well.

%

The remainder of the paper is organized as follows. Since cell decomposition is an important ingredient of the proofs throughout this paper, we present the necessary background on cell decomposition in Section \ref{section: cell decomposition}. Sections \ref{section: auxiliary 1} and \ref{section: auxiliary 2} contain several auxiliary results that will be needed in the later sections. Finally, Theorems \ref{thm:Zvar} and \ref{thm:Qpvar} will be proven in Sections \ref{section: Z} and \ref{section: K}.


\section{Preliminaries on cells and cell decomposition}\label{section: cell decomposition}
\label{cellsection}

In this section we will restate the Clustered Cell Decomposition Theorem from \cite{cham-cubi-leen}, 
followed by a more informal discussion where we will also introduce some further definitions. The results in this section are valid for any $P$-minimal structure $(K, \Gamma_K; \cL_2)$. 
\\\\
We will need the following notation.
For any $n,m \in \N\backslash\{0\}$, define $Q_{n,m}$ to be the set
\[Q_{n,m} = \{x \in K^\times \mid \ord(x) \equiv 0 \textnormal{ mod } n\,\wedge\,\ac_m(x) = 1\}.\]
where $\ac_m$ is the standard angular component map $K \to (\mathcal{O}_K)^{\times}/(\mathcal{M}_K)^m \cup \{0\}$.

\begin{theorem}[Clustered Cell Decomposition]\label{thm:celldecomposition}
Let $X \subseteq S \times K$ be a set definable in a $P$-minimal structure $(K, \Gamma_K; \cL_2)$. Then there exist $n,m \in \N\backslash\{0\}$ and a finite partition of $X$ into definable sets $X_i \subseteq S_i \times K$ of one of the following forms
\begin{itemize}
\item[(i)] Classical cells 
\[X_i= \{ (s,t) \in S_i \times K \mid \alpha_{i}(s) \ \square_1 \ \ord(t-c_{i}(s)) \ \square_2 \ \beta_{i}(s) \wedge t - c_{i}(s) \in \lambda_{i} Q_{n,m} \},\]
where $\alpha_{i}, \beta_{i}$ are definable functions $S_i \to \Gamma_K$, the squares $\square_1,\square_2$ may denote either $<$ or    $\emptyset$ (i.e. `no condition'), $\lambda_{i} \in K$. The center $c_{i}\colon  S_i \to K$ is a definable function (which may not be unique). 
\item[(ii)] Regular clustered cells $X_i=C_{i}^{\Sigma_{i}}$ of order $k_{i}$. \item[] Let $\sigma_1, \ldots, \sigma_{k_{i}}$ be (non-definable) sections of the definable multi-ball $\Sigma_{i} \subseteq S_i \times K$, such that for each $s \in S_i$, the set $\{\sigma_1(s), \ldots, \sigma_{k_{i}}(s)\}$ contains representatives of all $k_{i}$ disjoint balls covering $(\Sigma_{i})_s$. Then $X_{i}$ partitions as
\[X_{i} = C_{i}^{\sigma_1} \cup \ldots \cup C_{i}^{\sigma_{k_{i}}},\]
where each set $C_{i}^{\sigma_l}$ is of the form
\[ C_{i}^{\sigma_l} = \{(s,t) \in S_i \times K \mid \alpha_{i}(s) \ < \ \ord(t-\sigma_l(s)) \ < \ \beta_{i}(s) \wedge t-\sigma_l(s) \in \lambda_{i} Q_{n,m}\}.\]
Here $\alpha_{i}, \beta_{i}$ are definable functions $S_i \to \Gamma_K$, $\lambda_{i} \in K \backslash \{0\}$, and $\ord\, \alpha_{i}(s) \geqslant \ord \,\sigma_l(s)$ for all $s \in S_i$. Finally, we may suppose no section of $\Sigma_i$ is definable.  
\end{itemize}
\end{theorem}
Readers will probably be most familiar with the \emph{classical} cell decomposition theorem for $p$-adic semi-algebraic sets as it was originally proven by  Denef \cite{denef-86} (and then extended to the sub-analytic setting by  Cluckers \cite{clu-2003}). This type of cell is what we will refer to as \emph{classical} cells (see part (i) of Theorem \ref{thm:celldecomposition}). However, as stated before, in structures that do not admit definable Skolem functions, there are definable sets that cannot fully be partitioned into definable sets of this form by a result of Mourgues \cite{mou-09}.

Note that the issue here is definability, rather than geometry: we extend the notion of cells to include clustered cells, but geometrically these sets have a structure that is identical to that of (finite unions of) classical cells. Let us take a moment to explain some terms. We will keep this discussion informal, for technical details we refer to \cite{cham-cubi-leen}.
We use the tree representation of valued fields to visualize the structure of cells. Using this representation, we can visualize the fibers (for fixed values of $s \in S$) of a classical cell.  The different tree structures that can occur depend mainly on whether $\lambda_i = 0$ (0-cell), $\lambda_i \neq 0$ (1-cell), and on the value of $\square_2$.
\\\\
\begin{minipage}{0.5 \linewidth} A \emph{classical cell} (as in (i)) is often denoted as $C^{c_i}$. Here $c_i$ refers to the center, and $C$ refers to the rest of the description of the cell (which we sometimes refer to as a \emph{cell condition}.) For each value $s$ of the parameter set $S$, we denote the corresponding fiber as $C^{c_i(s)}:= (C^{c_i})_s$. It is one such fiber that is depicted on the right (for three different cases).
\\\\
 In these pictures, the grey triangles represent balls $C^{c_i(s), \gamma}:=\{t\in C^{c_i(s)} \mid  \ord (t-c_i(s)) = \gamma\}$. Such balls are what we will refer to as \emph{leaves}, and cells can be seen as a union of such leaves, for values of $\gamma$ as restricted by the description of the cell. We call $\gamma$ the \emph{height} of the leaf.
\end{minipage}
\begin{minipage}{0.1 \textwidth}
\end{minipage}
\begin{minipage}{0.4 \textwidth}
\begin{tikzpicture}[scale = 0.8]
\node[{above}] at (0,5.3){\footnotesize $c_i(s)$};
\draw [fill] (0,5.3) circle [radius = 1.25pt];

\node [{below}] at (0,-0.3){\footnotesize 0-cell};

\draw (2.25,5.3) -- (3,4.25) -- (3.75,5.3) -- cycle;
\draw (3,0) -- (3,4.25);
\draw [dotted, line width=0.75pt] (3,-0.3) -- (3,0);
\draw (3,3.5) -- (2.25,4.25);
\draw (3,2.75) -- (2.25, 3.5);
\draw (3,2) -- (2.25,2.75);
\draw (3,1.25) -- (2.25,2);
\draw (3,0.5) -- (2.25,1.25);
\path [fill = gray] (2.25,4.25) -- (1.6,4.6) -- (1.9,4.9) -- cycle;
\path [fill = gray] (2.25,3.5) -- (1.6,3.85) -- (1.9,4.15) -- cycle;
\path [fill = gray] (2.25,2.75) -- (1.6,3.1) -- (1.9,3.4) -- cycle;
\path [fill = gray] (2.25,2) -- (1.6,2.35) -- (1.9,2.65) -- cycle;
\path [fill = gray] (2.25,1.25) -- (1.6,1.6) -- (1.9,1.9) -- cycle;
\draw [dashed] (1.5,4.05) -- (4,4.05);
\draw [dashed] (1.5,0.25) -- (4,0.25);
\node[{right}] at (4,4.05){\footnotesize$\beta(s)$};
\node[{right}] at (4,0.25){\footnotesize$\alpha(s)$};
\node[{above}] at (3,5.3){\footnotesize$c_i(s)$};
\draw [fill] (3,5.3) circle [radius = 1.25pt];
\draw [decorate,decoration={brace,amplitude=4pt,mirror,raise=2pt},yshift=0pt] (3,0.5) -- (3,1.25) node [black,midway,xshift=9pt] {\tiny$n$};
\draw [decorate,decoration={brace,amplitude=4pt,mirror,raise=1.5pt},yshift=0pt, rotate=45] (2.47,-1.68) -- (2.47,-0.71) node [black,midway,xshift=5pt,yshift=6.5pt] {\tiny$m$};
\node [{above}] at (4.5,5.8){\footnotesize equivalence class $B_{\rho}(c_i(s))$};
\path[->] (4.5,5.9)  edge [bend left=50] (3.7,5.1);
\node [{below}] at (0.9,3.25){\footnotesize leaves};
\path[->] (0.9,3.17)  edge [bend left=50] (1.7,3.25);

\node [{below}] at (3,-0.3){\footnotesize 1-cell};
\node [{below}] at (3,-0.65){\footnotesize$\square_1 = \square_2 = <$};

\draw (7,0) -- (7,4.8);
\draw [dotted, line width=0.75pt] (7,-0.3) -- (7,0);
\draw [dotted, line width=0.75pt] (7,4.8) -- (7,5.3);
\draw (7,4.25) -- (6.25,5);
\draw (7,3.5) -- (6.25,4.25);
\draw (7,2.75) -- (6.25, 3.5);
\draw (7,2) -- (6.25,2.75);
\draw (7,1.25) -- (6.25,2);
\draw (7,0.5) -- (6.25,1.25);
\path [fill = gray] (6.25,5) -- (5.6,5.35) -- (5.9,5.65) -- cycle;
\path [fill = gray] (6.25,4.25) -- (5.6,4.6) -- (5.9,4.9) -- cycle;
\path [fill = gray] (6.25,3.5) -- (5.6,3.85) -- (5.9,4.15) -- cycle;
\path [fill = gray] (6.25,2.75) -- (5.6,3.1) -- (5.9,3.4) -- cycle;
\path [fill = gray] (6.25,2) -- (5.6,2.35) -- (5.9,2.65) -- cycle;
\path [fill = gray] (6.25,1.25) -- (5.6,1.6) -- (5.9,1.9) -- cycle;
\draw [dashed] (5.5,0.25) -- (8,0.25);
\node[{right}] at (8,0.25){\footnotesize$\alpha(s)$};
\node[{above}] at (7,5.3){\footnotesize$c_i(s)$};
\draw [fill] (7,5.3) circle [radius = 1.25pt];
\node [{below}] at (4.9,2.5){\footnotesize$C^{c_i(s),\gamma}$};
\path[->] (4.9
,2.42)  edge [bend left=50] (5.7,2.5);
\draw [fill] (7,1.25) circle [radius = 1.25pt];
\node[{right}] at (7,1.25){\footnotesize$\gamma$};

\node [{below}] at (7,-0.3){\footnotesize 1-cell};
\node [{below}] at (7,-0.65){\footnotesize$\square_1 = <, \square_2 = \emptyset$};
\end{tikzpicture}
\end{minipage}
\\\\
The function $c_i(s)$ is what we call the \emph{center} of a cell. 
When $\square_2 = \ <$, such a center is not unique, and hence we can define (for every $s$) the \emph{equivalence class} of all elements $a\in K$ such that $C^{c_i(s)}=C^a$. Such an equivalence class is a ball with center $c_i(s)$ which we denote by $B_{\rho}(c_i(s))$.\\
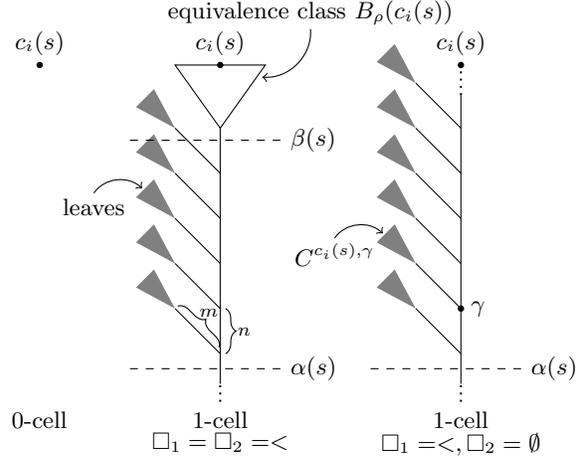
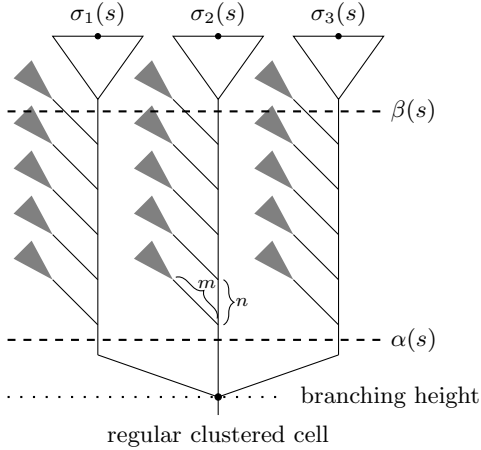
\begin{minipage}{0.45\textwidth}
\begin{tikzpicture}[scale=0.8]
\draw (2.25,5.3) -- (3,4.25) -- (3.75,5.3) -- cycle;
\draw (3,-1) -- (3,4.25);
\draw (3,3.5) -- (2.25,4.25);
\draw (3,2.75) -- (2.25, 3.5);
\draw (3,2) -- (2.25,2.75);
\draw (3,1.25) -- (2.25,2);
\draw (3,0.5) -- (2.25,1.25);
\path [fill = gray] (2.25,4.25) -- (1.6,4.6) -- (1.9,4.9) -- cycle;
\path [fill = gray] (2.25,3.5) -- (1.6,3.85) -- (1.9,4.15) -- cycle;
\path [fill = gray] (2.25,2.75) -- (1.6,3.1) -- (1.9,3.4) -- cycle;
\path [fill = gray] (2.25,2) -- (1.6,2.35) -- (1.9,2.65) -- cycle;
\path [fill = gray] (2.25,1.25) -- (1.6,1.6) -- (1.9,1.9) -- cycle;
\draw (0.25,5.3) -- (1,4.25) -- (1.75,5.3) -- cycle;
\draw (1,0) -- (1,4.25);
\draw (1,3.5) -- (0.25,4.25);
\draw (1,2.75) -- (0.25, 3.5);
\draw (1,2) -- (0.25,2.75);
\draw (1,1.25) -- (0.25,2);
\draw (1,0.5) -- (0.25,1.25);
\path [fill = gray] (0.25,4.25) -- (-0.4,4.6) -- (-0.1,4.9) -- cycle;
\path [fill = gray] (0.25,3.5) -- (-0.4,3.85) -- (-0.1,4.15) -- cycle;
\path [fill = gray] (0.25,2.75) -- (-0.4,3.1) -- (-0.1,3.4) -- cycle;
\path [fill = gray] (0.25,2) -- (-0.4,2.35) -- (-0.1,2.65) -- cycle;
\path [fill = gray] (0.25,1.25) -- (-0.4,1.6) -- (-0.1,1.9) -- cycle;
\draw (4.25,5.3) -- (5,4.25) -- (5.75,5.3) -- cycle;
\draw (5,0) -- (5,4.25);
\draw (5,3.5) -- (4.25,4.25);
\draw (5,2.75) -- (4.25, 3.5);
\draw (5,2) -- (4.25,2.75);
\draw (5,1.25) -- (4.25,2);
\draw (5,0.5) -- (4.25,1.25);
\path [fill = gray] (4.25,4.25) -- (3.6,4.6) -- (3.9,4.9) -- cycle;
\path [fill = gray] (4.25,3.5) -- (3.6,3.85) -- (3.9,4.15) -- cycle;
\path [fill = gray] (4.25,2.75) -- (3.6,3.1) -- (3.9,3.4) -- cycle;
\path [fill = gray] (4.25,2) -- (3.6,2.35) -- (3.9,2.65) -- cycle;
\path [fill = gray] (4.25,1.25) -- (3.6,1.6) -- (3.9,1.9) -- cycle;
\draw [dashed, thick] (-0.5,4.05) -- (5.7,4.05);
\draw [dashed, thick] (-0.5,0.25) -- (5.7,0.25);
\node[{above}] at (1,5.3){\footnotesize$\sigma_1(s)$};
\draw [fill] (1,5.3) circle [radius = 1.25pt];
\node[{above}] at (3,5.3){\footnotesize$\sigma_2(s)$};
\draw [fill] (3,5.3) circle [radius = 1.25pt];
\node[{above}] at (5,5.3){\footnotesize$\sigma_3(s)$};
\draw [fill] (5,5.3) circle [radius = 1.25pt];
\draw (1,0) -- (3,-0.7);
\draw (5,0) -- (3,-0.7);
\node[{below}] at (3,-1){\footnotesize regular clustered cell};
\node[{right}] at (5.7,0.25){\footnotesize $\alpha(s)$};
\node[{right}] at (5.7,4.05){\footnotesize $\beta(s)$};
\draw [decorate,decoration={brace,amplitude=4pt,mirror,raise=2pt},yshift=0pt] (3,0.5) -- (3,1.25) node [black,midway,xshift=9pt] {\tiny$n$};
\draw [decorate,decoration={brace,amplitude=4pt,mirror,raise=1.5pt},yshift=0pt, rotate=45] (2.47,-1.68) -- (2.47,-0.71) node [black,midway,xshift=5pt,yshift=6.5pt] {\tiny$m$};
\draw [fill] (3,-0.7) circle [radius = 1.5pt];
\draw [loosely dotted, thick] (-0.5,-0.7) -- (4.1,-0.7);
\node[{right}] at (4.2,-0.7){\footnotesize branching height};
\end{tikzpicture}
\end{minipage}
\begin{minipage}{0.55 \linewidth}
In \emph{clustered cells}, we keep the same geometric notion of cells, but we will now make use of these equivalence classes to define centers, rather than using definable functions (which may not always exist). \\\\ The multi-ball $\Sigma_i$ from part (ii) of the above definition is a set whose fibers consist of $k_i$ balls, each corresponding to an equivalence class of centers. Hence, when we replace the center of a cell by such a set $\Sigma_i$, we obtain a set (denoted as $C^{\Sigma_i}$) that geometrically has the structure of a union of $k_i$ (disjoint) classical cells that only differ in their description by the use of different centers, as shown on the left. 
\end{minipage}
\\\\\\
Note that the sections $\sigma_j$ are picking representatives $\sigma_j(s)$ from each equivalence class, yet are not necessarily definable (as it may be that no definable section exists.) We will still continue using the notation introduced above, but write $\sigma_i$ rather than $c_i$ to stress the fact that representatives $\sigma_i$ need not be definable.\\\\
In order to work with such clustered cells, one sometimes needs to take the structure of the set $\Sigma_i$ into consideration. For example, for every fiber $(\Sigma_i)_s$, the \emph{branching heights} are the values $\gamma \in \Gamma_K$ for which there exist representatives $\sigma_1(s), \sigma_2(s)$ of different equivalence classes, such that $\ord(\sigma_1(s) - \sigma_2(s)) = \gamma$, as shown on the picture above. If a cell is \emph{large} (i.e., each cell fiber has leaves at more than one height), then all branching heights are smaller than $\alpha(s)$, for every $s \in S$. \\\\
Visual representations like the pictures shown above are a representation of one of the fibers of a cell, for a fixed value of $s$. However, these pictures allow us to deduce the general structure of the cell as well, because of the condition of \emph{regularity}. We will informally explain what consequence this condition has for the general structure of the fibers of $\Sigma_i$.\\\\
Suppose that $\Sigma_i$ is a multi-ball of order $k$. Given representatives $\sigma_1(s),\ldots, \sigma_k(s)$ of each equivalence class in $(\Sigma_i)_s$, one can look at the finite tree $T_s$ they induce, which will result in a picture like in the figure below. The regularity of the cell establishes that for all $s\in S$, such finite trees are all ``isomorphic'', meaning that there exists an order-preserving bijection between each two trees $T_s$ and $T_{s'}$. 
\\
\begin{minipage}{0.5 \textwidth}

The following notion is important for studying the structure of the fibers of $\Sigma_i$.
If $\gamma_1 > \gamma_2 > \ldots > \gamma_d$ are the $d$ highest branching heights, then we can assign a $d$-\emph{signature} $(k_1, \ldots, k_d)$ to each point of $(\Sigma_i)_s$, as illustrated by the picture on the right.  In the tree shown here, $\sigma_1$ has 3-signature $(3,1,2)$ and $\sigma_2$ has 3-signature $(2,3,2)$. 
 \end{minipage}
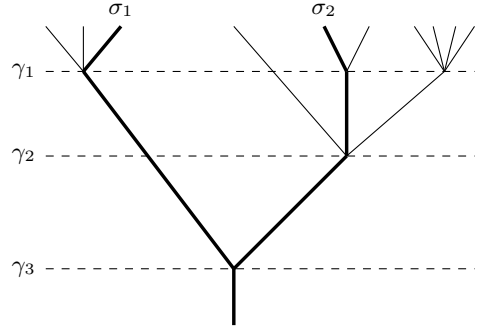
\begin{minipage}{0.15 \textwidth}
\end{minipage}
\begin{minipage}{0.35 \textwidth}
\begin{tikzpicture}[yscale = 0.75]
\draw [dashed] (0.5,1) -- (6.2,1);
\draw [dashed] (0.5,3) -- (6.2,3);
\draw [dashed] (0.5,4.5) -- (6.2,4.5);
\node[{left}] at (0.5,1){\footnotesize$\gamma_3$};
\node[{left}] at (0.5,3){\footnotesize$\gamma_2$};
\node[{left}] at (0.5,4.5){\footnotesize$\gamma_1$};

\draw [line width=1.3pt] (3,0) -- (3,1);
\draw [line width=1.3pt] (3,1) -- (1,4.5);
\draw [line width=1.3pt] (1,4.5) -- (1.5,5.3);
\draw (1,4.5) -- (1,5.3);
\draw (1,4.5) -- (0.5,5.3);

\draw [line width=1.3pt] (3,1) -- (4.5,3);
\draw [line width=1.3pt] (4.5,3) -- (4.5,4.5);
\draw [line width=1.3pt] (4.5,4.5) -- (4.2,5.3);
\draw (4.5,4.5) -- (4.8,5.3);
\draw (4.5,3) -- (3,5.3);
\draw (4.5,3) -- (5.8,4.5);
\draw (5.8,4.5) -- (5.4,5.3);
\draw (5.8,4.5) -- (5.65,5.3);
\draw (5.8,4.5) -- (5.95,5.3);
\draw (5.8,4.5) -- (6.2,5.3);

\node[{above}] at (1.5,5.3){\footnotesize$\sigma_1$};
\node[{above}] at (4.2,5.3){\footnotesize$\sigma_2$};
\end{tikzpicture}
\end{minipage}\\
\\\\
\noindent For more details we refer to \cite{cham-cubi-leen}. In particular, the following definitions may be of relevance to the contents of this paper: Definitions 1.4 (leaf), 4.1 (equivalence class), 4.3 (branching height), 4.4 (signature), 5.3 (large/small cells), 6.2 (regular clustered cell).
\\\\
For the current paper, we will need to slightly strengthen this cell decomposition result, adding an extra condition on the structure of the set of centers $\Sigma$. The exact result is formulated in the following theorem, that can also be found in \cite[Lemma 3.2]{cham-cubi-leen-note}

\begin{theorem}\label{thm:symmetrictrees}
Let $X \subseteq S \times K$ be a definable set in a $P$-minimal structure $(K, \Gamma_K; \cL_2)$. Then $X$ can be partitioned as a finite union of classical cells and regular clustered cells. Each regular clustered cell $C^\Sigma$ of order $k$ satisfies the following two properties.
\begin{itemize}
\item[(i)] The set $\Sigma$ does not admit any definable sections.
\item[(ii)] There exists $d \in \NN$ such that for every $s \in S$, $\Sigma_s$ has exactly $d$ branching heights and when $d \geq 1$, there exists $k_1,\ldots, k_d \in \NN$, such that all elements of each $\Sigma_s$ have $d$-signature $(k_1, \ldots, k_d)$.
\end{itemize}
\end{theorem}

\begin{proof}
Because of Theorem \ref{thm:celldecomposition} we can assume that $X$ partitions into classical cells and regular clustered cells $C^\Sigma$ of order $k$, that already satisfy condition (i). For condition (ii) we can assume that $X = C^\Sigma$. Since the tree structures of all the fibers of $\Sigma$ are isomorphic, we can, after a finite partitioning of $S$, assume that for each $s, s' \in S$, $\Sigma_s$ and $\Sigma_{s'}$ essentially look the same. What we mean by this is that the number of branching heights is the same and if we were to pick representatives for equivalence classes of $(\Sigma_i)_s$ and $(\Sigma_i)_{s'}$, then there would exist a bijection between these sets of representatives that preserves all $d$-signatures. This already establishes the existence of $d$ from condition (ii).

Now if $k=1$, then condition (ii) is automatically satisfied, so we may assume that $k>1$, which in turn implies that $d\geqslant 1$. For  each $l \in \N$, write $(k_1(c), \ldots, k_l(c))$ for the $l$-signature of $c \in \Sigma_s$.
If $C^\Sigma$ does not yet satisfy condition (ii), then there exists some $s \in S$ for which the $d$-signature is not fixed on $\Sigma_s$, hence for every $s\in S$, $\Sigma_s$ will contain elements with at least two different signatures. In this case we will give an explicit decomposition of $C^\Sigma$ into regular clustered cells that satisfy both conditions.

First, partition $\Sigma$ in sets $\Sigma_{(l_1)}$, for $l_1 \in \{1, \ldots, q_K\}$, which are defined as
\[\Sigma_{(l_1)}:= \{(s,c) \in \Sigma \mid k_1(c) = l_1\}.\]
Note that some of these sets may be empty. This induces a partition of $C^{\Sigma}$ into the union of the regular clustered cells $C^{\Sigma_{(l_1)}}$. (It should be clear that the uniformity of the tree structure is preserved. Further, since the tree of $(\Sigma_{(l_1)})_s$ is a pruning of the original tree of $\Sigma$, and no new branching heights are introduced, we still have that all branching happens below $\alpha(s)$.)

This process can now be repeated inductively. If we fix a clustered cell $C^{\Sigma_{(l_1)}}$, the 1-signature is fixed. This clustered cell can now be partitioned into cells $C^{\Sigma_{(l_1,l_2)}}$, where $C^{\Sigma_{(l_1,l_2)}}$ is defined as
\[\Sigma_{(l_1,l_2)}:= \{(s,c) \in \Sigma_{(l_1)} \mid k_2(c) = l_2\},\]
again for $l_2 \in \{1, \ldots, q_K\}$. We can repeat the process until we have a partition of $C^\Sigma$ into regular clustered cells $C^{\Sigma_{(l_1, \ldots, l_d)}}$ that satisfy conditions (i) and (ii).
\end{proof}

For a regular clustered cell $C^\Sigma$ satisfying the two conditions from the previous theorem, the tuple $(k_1, \ldots, k_d)$ will be called the \emph{tree type} of $C^\Sigma$.
\\\\
The results mentioned so far in this section are about cells in $S \times K$, where the last variable is of the $K$-sort. As we are working in two-sorted structures, we will also occasionally need to work with cells where the last variable is of the value-group sort. Such cells will be called $\Gamma$-cells. We recall the following result from \cite{cubi-leen}:
\begin{theorem}[\cite{cubi-leen}, Proposition 2.4]\label{thm:recti}
Let $f: X \subseteq S \times \Gamma_K \to \Gamma_K$ be definable in a $P$-minimal structure $(K, \Gamma_K; \Lm_2)$. There exists a finite partition of $X$ in $\Gamma$-cells $C$ of the form
\[C = \left\{ (s,\gamma) \in D \times \Gamma_K \left| \begin{array}{l} \alpha(s) \ \square_1 \ \gamma \ \square_2 \ \beta(s),\\ \gamma \equiv k \mod n \end{array}\right\}\right.,\]
where $D$ is a definable subset of $S$, $\alpha, \beta$ are definable functions $D \to \Gamma_K$, $k,n, \in \N$ and the squares $\square_i$ may denote $<$ or no condition. On each such cell, the function $f_{|C}$ has the form
\[f_{|C}(s,\gamma) = a\left( \frac{\gamma -k}{n}\right) + \delta(s),\]
where $a \in \Z$ and $\delta$ is a definable function $D \to \Gamma_K$. 
\end{theorem}


\section{Auxiliary results on multi-balls over the value group}\label{section: auxiliary 1}

In this section $(K, \Gamma_K; \cL_2)$ will be a (general) $P$-minimal structure. The main result of this section is Proposition \ref{prop:Gamma-multiball}, which holds under the assumtion of relative $P$-minimality.

\begin{definition}\label{def:relPmin} A structure $(K, \Gamma_K; \cL_2)$ is called \emph{relative $P$-minimal} if for all $n\geqslant 0$, every $\cL_2$-definable subset of $K\times \Gamma_K^n$ is definable in $\cL_{\text{ring},2}$. 
\end{definition}

\begin{definition}\label{def:evp} A structure $(K,\Gamma_K; \cL_2)$ has the \emph{extreme value property} if for every closed and bounded subset $U\subseteq K$ and every definable continuous function $f\colon U \to \Gamma_K$, $f(U)$ admits a maximal value.  
\end{definition}

The following is a reformulation of Theorem 4.1 in \cite{Dar-Hal-2017}. 

\begin{theorem}[Darni\`ere-Halupczok]\label{thm:relPmin} Assume that $(K, \Gamma_K; \cL_2)$ is $P$-minimal and satisfies the extreme value property. Then $(K,\Gamma_K; \cL_2)$ is relative $P$-minimal.
\end{theorem}

\begin{remark}\label{rem:relPmin} Note that every $P$-minimal expansion of a $p$-adic field satisfies the extreme value property. Indeed, this holds more generally for any $P$-minimal field having value group $\ZZ$. Therefore, all the results proven in this section for relative $P$-minimal structures, will hold in particular for $P$-minimal expansions of $p$-adic fields.
\end{remark}

\subsection{A finiteness result}

The purpose of this subection is to show the following theorem:

\begin{theorem}\label{thm:compactify} Let $(K, \Gamma_K; \cL_2)$ be a relative $P$-minimal structure and let  $A \subseteq S \times \Gamma_K \times K$ be a definable multi-ball of order $k$ with fibers of the form
\[A_{s,\gamma} = \text{union of $k$ disjoint balls in } \pball.\]
Then there exists a uniform bound $N \in \NN$, such that for every $s \in S$, 
\[ \#\{B \in \pball \mid \exists \gamma \in \Gamma_K: B \subseteq A_{s,\gamma}\} < N.\]
\end{theorem}

\begin{remark}
This theorem holds also for multi-balls for which the $k$ balls in each fiber are in $\BB_\eta$ for some $\eta \in \Gamma_K$. 
\end{remark}

Before we can give the proof of Theorem \ref{thm:compactify}, we will need some preliminary results. The following lemma is due to Haskell and Macpherson \cite[Remark 3.4]{has-mac-97}.

	\begin{lemma}\label{lem:locost} Let $(K, \Gamma_K; \cL_2)$ be a $P$-minimal structure and let $g \colon D \subseteq K\to \Gamma_K$ be a definable function. Then there exists a finite set $D'$ such that $g$ is locally constant on $D\backslash D'$. 
\end{lemma}

	\begin{lemma}\label{lem:finiteimage} Let $(K, \Gamma_K; \cL_2)$ be a $P$-minimal structure and let $f:\Gamma_K\to K$ be a definable function. Then $f$ has finite image. 
\end{lemma}

\begin{proof} We apply $\Gamma$-cell decomposition (Theorem \ref{thm:recti}) to the inverted graph of $f$, that is, to the set
		\[
		A:=\{(x,\gamma)\in K\times\Gamma_K: f(\gamma)=x\}.
		\]
		It is sufficient to show that on each $\Gamma$-cell $C$ of such a decomposition, the projection onto the first coordinate is finite. Let $C$ be a given $\Gamma$-cell of the decomposition,
		\[
		C:= \left\{(x,\gamma)\in D\times \Gamma_K \Bigm\vert \begin{array}{l} \alpha(x)\ \square_1' \ \gamma \ \square_2' \ \beta(s), \\ \gamma \equiv k\mod n' \end{array}\right\},
		\]
		where $D$ is a definable subset of $K$, $\alpha, \beta$ are definable functions $D\to\Gamma_K$ and $k, n'\in \NN$. We may furthermore assume that $D$ is a $K$-cell over $\emptyset$, that is, $D$ is of the form
		\[
		D = \left\{x\in K \mid \gamma_1\ \square_1 \ \ord(x-c) \ \square_2 \ \gamma_2\wedge \ x-c \in \lambda Q_{m,n}\right\},
		\]
		where $\gamma,\gamma_2\in \Gamma_K$, $c,\lambda\in K$, and $m,n\in \NN$. We will show that $\lambda=0$, by deriving a contradiction from $\lambda \neq 0$. This result will imply directly that the projection onto the first coordinate is finite. So let us assume that $\lambda\neq 0$. Consider the following two cases.
		
		\
		
		\textbf{Case 1:} Suppose that $\square_1'$ (resp.\ $\square_2'$) equals ``no condition''. Pick distinct $x, y\in D$ and $\gamma\in \Gamma_K$ such that $\gamma\equiv k \ \text{mod}\ n'$ and $\gamma \in D_x \cup D_y$. Because of the assumption, we can do this by taking $\gamma$ small enough (resp.\ big enough), i.e., $\gamma<\min(\beta(x),\beta(y))$ (resp.\ $\gamma>\max(\alpha(x),\alpha(y))$). This contradicts the assumption that $f$ is a function, since $\gamma$ will have two images $x$ and $y$.
		
		\
		
		\textbf{Case 2:} Suppose that both $\square_1'$ and $\square_2'$ are `$<$'. By Lemma \ref{lem:locost}, there exist distinct $x,y\in D$ such that $\alpha(x)=\alpha(y)$ and $\beta(x)=\beta(y)$. As before, this contradicts the assumption that $f$ is a function, since any $\gamma$ such that $\alpha(x)<\gamma<\beta(x)$ will have both $x$ and $y$ as images.
\end{proof}

\begin{lemma}\label{lemma:localconstant}
Let $(K, \Gamma_K; \cL_2)$ be a relative $P$-minimal structure and let $\alpha\colon D \subseteq K \to \Gamma_K$ be a definable function. Then there exists a finite set $D'$ and constants $c_1, \ldots, c_l \in K$, and a partition of $D \setminus D'$ into $l$ $1$-cells $C_i^{c_i}$, such that the function $\alpha$ is constant on each of the leaves $C_i^{c_i, \gamma}$.
\end{lemma}
\begin{proof}
By relative $P$-minimality, the inverted graph of $\alpha$ is an $\cL_{\text{ring},2}$-definable set, which can be partitioned as a finite union of classical cells of the form
\[
C:= \left\{ (\gamma,x) \in \Gamma_K \times K \Bigm\vert 
\begin{array}{cccccc }
a(\gamma) &\square_{11} &\ord(x-c(\gamma)) & \square_{12} & b(\gamma)&\wedge \ x-c(\gamma) \in \lambda Q_{n,m}\\
c_1 & \square_{21}  & \gamma &  \square_{22} & c_2 &\wedge\ \gamma \equiv k \mod n'\\
\end{array}
 \right\},
 \]
where $c\colon \Gamma_K \to K$ is a definable function. Moreover, by Lemma \ref{lem:finiteimage} we know that $c$ has finite image, hence we may as well assume that $c(\gamma)$ is in fact constant on each cell. Note that, if $\lambda =0$ for some cell $C$, then $C$ only contains a single point $(\alpha(c),c)$. We take $D'$ to be the union of these values $c \in K$.

Let us show that the projection of a cell $C$ with constant center $c(\gamma) = c$ and $\lambda \neq 0$, onto the second variable, can be written as a finite union of cells $C_i^c \subseteq D \setminus D'$ (that is, all cells are centered at $c$). 
Let $Z$ denote the projection of $C$ onto the second variable, and consider the set $Y:=\{\ord(x-c) \mid x\in Z\}$. The set $Y$ can be partitioned into finitely many $\Gamma$-cells $Y_1,\ldots, Y_r$. The reader can check that for $i\in\{1,\ldots,r\}$, the sets
\[
C_i^c:= \left\{ x \in K\mid 
\begin{array}{cccccc }
\ord(x-c)\in Y_i & \wedge \ x-c \in \lambda Q_{n,m}\\
\end{array}
 \right\}
 \]
form a cell decomposition of $Z$ with cells centered at $c$. 

Doing this for all cells for which $\lambda \neq 0$, gives a partition of $D\setminus D'$. If $x \in C_i^c$, then the value of $\alpha(x)$ equals the unique $\gamma$ for which $a(\gamma) \ \square_{11} \ \ord(x-c) \ \square_{12} \ b(\gamma)$. Hence, $\alpha(x)$ is constant on leaves of $C_i^{c}$.
\end{proof}

\begin{lemma}\label{lemma:partition}
Let $(K, \Gamma_K; \cL_2)$ be a relative $P$-minimal structure and let $\alpha_1, \alpha_2\colon D \subseteq K \to \Gamma_K$ be two definable functions. Then there exists a finite set $D'$, finitely many constants $c_1, \ldots, c_l \in K$ and a partition of $D\setminus D'$ into $l$ cells $C_j^{c_j}$, such that on each cell, both $\alpha_{1}$ and $\alpha_{2}$ have constant value on each leaf $C_j^{c_j, \gamma}$.
\end{lemma}
\begin{proof}
Applying Lemma \ref{lemma:localconstant} yields two partitions of $D$:
\begin{align*}
D =  D'_{1} \cup \bigcup_{i=1}^l D_i^{d_i} = D'_{2} \cup \bigcup_{\iota=1}^{l'} E_\iota^{e_\iota},
\end{align*}
such that the $D'_{j}$ are finite sets, and $D_i^{d_i}$, resp.\ $E_\iota^{e_\iota}$ are 1-cells such that ${\alpha_{1}}$, resp.\ ${\alpha_{2}}$ have constant value on leaves of $D_i^{d_i}$, resp.\ $E_\iota^{e_\iota}$.

A refinement of both partitions can be found by considering intersections $D_i^{d_i} \cap E_\iota^{e_\iota}$, $D'_{1} \cap E_\iota^{e_\iota}$ and $D'_{2} \cap D_i^{d_i}$. The intersection of a point and a cell can either be empty or a point. The intersection of two cells $D_i^{d_i} \cap E_\iota^{e_\iota}$ is either empty, or a definable set that can once again be partitioned as a finite union of points and $1$-cells $C_j^{c_j}$. In order to finish the proof, we need to check that, for any $\gamma_0 \in \Gamma_K$, there exist $\gamma, \gamma' \in \Gamma_K$ such that
\begin{equation}\label{eq:intersect} 
C_j^{c_j, \gamma_0} \subseteq D_i^{d_i, \gamma} \cap E_\iota^{e_\iota, \gamma'}.
\end{equation}
Indeed, if this holds then $\alpha_1$ and $\alpha_2$ will have constant value on the leaves of $C_j^{c_j}$ as required.

\

We will show that there exists $\gamma \in \Gamma_K$, such that $C_j^{c_j, \gamma_0} \subseteq D_i^{d_i, \gamma}$. The same argument will allow us to find $\gamma' \in \Gamma_K$ such that $C_j^{c_j, \gamma_0} \subseteq E_\iota^{e_\iota, \gamma'}$. These two statements together imply \eqref{eq:intersect}.
Since $C_j^{c_j, \gamma_0} \subseteq D_i^{d_i}$, we know that there must exist at least one leaf of $D_i^{d_i}$ that has nonempty intersection with $C_j^{c_j, \gamma_0}$. Let $L := \{ \rho \in \Gamma_K \mid D_i^{d_i, \rho} \cap C_j^{c_j, \gamma_0}\neq \emptyset\}$ be the set listing the heights of such leaves. We need to check that $L$ cannot contain more than one element. Note that, if $L$ has more than one element, then $C_j^{c_j, \gamma_0}$ contains elements from at least two different leaves of $D_i^{d_i}$, hence $C_j^{c_j, \gamma_0}$ also contains the smallest ball that contains these elements. Such a ball will always contain the center $d_i$, hence $d_i \in C_j^{c_j, \gamma_0}$, but $d_i \notin D_i^{d_i}$. This contradicts $C_j^{c_j, \gamma_0} \subseteq D_i^{d_i}$,
and therefore we can conclude that $L$ can only have a single element $\gamma$, which implies that $C_j^{c_j, \gamma_0} \subseteq D_i^{d_i, \gamma}$.
 \end{proof}

\begin{proof}[Proof of Theorem \ref{thm:compactify}:]
By compactness, it suffices to show that for every $s \in S$, there exists $N_s \in \NN$, and balls $B_{1}, \ldots, B_{N_s}$ from $\pball$, such that
\[
\bigcup_{\gamma} A_{s,\gamma} = B_1 \cup \ldots \cup B_{N_s}.
\] 
Fix $s \in S$, and consider the fiber $A_{s} \subseteq \Gamma_K \times K$. Reversing the order of the variables and applying $\Gamma$-cell decomposition, this set can be partitioned as a finite union of cells of the form
\[
C:= \left\{ (x,\gamma) \in D \times \Gamma_K \mid \alpha_1(x) \ \square_{11}  \ \gamma \  \square_{12} \ \alpha_{2}(x) \ \wedge \  \gamma \equiv \kappa \mod n' \right\},
\]
where $D$ is a semi-algebraic cell of the form 
\[
D:= \{x \in K \mid \gamma_{1} \ \square_{21} \ \ord(x-c) \  \square_{22} \ \gamma_2 \ \wedge \ \ x-c \in \lambda Q_{n,m}\},
\]
and $\gamma_i, \kappa \in \Gamma_K$; $\lambda, c\in K$; $m, n, n' \in \N$ and the $\alpha_i\colon D \to \Gamma_K$ are definable functions. By Lemma \ref{lemma:partition}, there is a finite set $D'$, finitely many constants $c_1, \ldots, c_l \in K$ and a partition of $D\setminus D'$ into $l$ cells $C_j^{c_j}$, such that on each cell, both $\alpha_{1}$ and $\alpha_{2}$ have constant value on each leaf $C_j^{c_j, \gamma}$.

Choose $r$ such that $k < q_K^r$. 
Fix one of the cells $C_j^{c_j}$, and consider a leaf $C_j^{c_j, \gamma_0}$ for some $\gamma_0 < 1-m_j-r$, where $m_j$ is as in the set $Q_{n_j,m_j}$, appearing in the cell condition $C_j$. Note that this leaf is the union of at least $q_K^r$ disjoint balls from $\pball$. Take some $x \in C_j^{c_j, \gamma_0}$, and choose $\gamma$ such that
\[
\alpha_1(x) \ \square_{11} \ \gamma  \ \square_{12} \ \alpha_{2}(x) \ \wedge \ \gamma \equiv \kappa \mod n'.
\] 
Then $x \in A_{s,\gamma}$. Lemma \ref{lemma:partition} implies that $\alpha_i(x) = \alpha_i(x')$ for any other $x' \in C_j^{c_j, \gamma_0}$, and hence $C_j^{c_j, \gamma_0} \subseteq A_{s, \gamma}$. This means that $A_{s,\gamma}$ must contain at least $q_K^r >k$ balls from $\pball$, which contradicts our assumption that $A_{s,\gamma}$ consists of $k$ balls. 

The only way this contradiction can be avoided is if the cells $C_j^{c_j}$ have no leaves $C_j^{c_j, \gamma_0}$ for which $\gamma_0 < 1-m_j-r$. This in turn implies that $D$ can only intersect a finite number of disjoint balls from $\pball$, since for $\gamma \geqslant 1$, all leaves $C_j^{c_j, \gamma}$ of the cell $C_j^{c_j}$ are contained within a single ball of $\pball$. Hence, the theorem follows.
\end{proof}

\subsection{Multiballs over the value group}

In this section we show that in relative $P$-minimal structures, definable multi-balls on $\pball$ over definable sets of the form $S\times \Gamma_K$ can be partitioned into finitely many definable sets which are multi-balls over $S$.

\begin{proposition}\label{prop:Gamma-multiball} Let $(K, \Gamma_K; \cL_2)$ be a relative $P$-minimal structure and let $X\subseteq S\times \Gamma_K$ be a definable set and $A\subseteq X\times K$ be a multi-ball of order $k$ on $\pball$. There is a finite set $W$ and a definable partition $X=\bigcup_{w\in W}X_w$ with $S_w:=\pi_S(X_w)$, such that for every $s\in S_w$, $A_s$ has constant fibers over $(X_w)_s$ (i.e., $A_{s,\gamma_1}=A_{s,\gamma_2}$ for all $\gamma_1,\gamma_2\in (X_w)_s$). 
\end{proposition}

\begin{proof}
By Theorem \ref{thm:compactify}, there is an integer $N_A$ such that for every $s$, there exists $N_s <N_A$, and balls $\{B_{1,s}, \ldots, B_{N_s,s}\} =: $ $\cB_s$ from $\pball$ (depending on $s$!), such that
\[
\bigcup_{\gamma\in X_s} A_{s,\gamma} = B_{1,s} \cup \ldots \cup B_{N_s,s}.
\] 
By partitioning $S$ into finitely many definable pieces, without loss of generality we may assume that the cardinality of $\cB_s$ is constant and equal to $N$ for all $s\in S$. 

\

For every $s\in S$, there are ${N}\choose{k}$ possible values for the fiber $A_{s,\gamma}$. Recall that a definably well-ordering $\lhd$ on $\Gamma_K$ is a linear ordering satisfying that every definable subset $Y\subseteq \Gamma_K$ has a $\lhd$-minimal element. Let $\lhd$ be the definably well-ordering on $\Gamma_K$ defined by 
\[x\lhd y \Leftrightarrow |x|<|y| \vee x=y \vee (-x = y \wedge y\geqslant 0).\]
To see that $\lhd$ is a definably well-ordering on $\Gamma_K$, note first that on $\Z$, $\lhd$ defines the well-ordering  
\[
0\lhd -1\lhd 1\lhd -2 \lhd 2 \lhd \cdots,   
\]
so, in particular, every $\cL_{\text{Pres}}$-definable subset of $\ZZ$ has a $\lhd$-minimal element. Since $(\Gamma_K,\cL_{\text{Pres}})\equiv(\ZZ, \cL_{\text{Pres}})$ (where $\cL_{\text{Pres}}$ denotes the Presburger language), the fact that every definable subset of $\Gamma_K$ is $\cL_{\text{Pres}}$-definable implies the desired property.

\

Let $\delta_1\colon S\to \Gamma_K$ be the definable function sending $s$ to $\min_\lhd(X_s)$, the minimal element with respect to the ordering $\lhd$. Setting $Z_1:=X$, we inductively define sets $Z_{i+1}\subseteq Z_i$ and functions $\delta_{i+1}\colon S\to \Gamma_K\cup\{\infty\}$ for $i\geqslant 1$ as follows:  
\[
Z_{i+1}:=\{(s,\gamma) \in Z_i\mid   A_{s,\gamma} \neq  A_{s,\delta_i(s)}\}
\]
\[
\delta_{i+1}(s):=
\begin{cases}
\min_\lhd(Z_{i+1})_{s} & \text{if } (Z_{i+1})_{s}\neq\emptyset \\ 
\infty & \text{ otherwise. }
\end{cases}
\]
The idea is to order the different configurations of $\cB_s$ appearing as fibers of $A_{s,\gamma}$ with respect to their minimal representatives in $X_s$. Doing this uniformly in $S$ and grouping the elements defining the same fiber is the idea behind the sets $Z_i$. Note that 
\[
X:=Z_1\supseteq Z_2 \supseteq \cdots \supseteq Z_{ {N}\choose{k} } \supseteq Z_{{{N}\choose{k}} +1}=\emptyset. 
\]
Moreover, for $s\in S$, if $\delta_{i}(s)\neq \infty$, the same holds for all $i'<i$. Therefore, the sets $S_1,\ldots,S_{{{N}\choose{k}}}$ with  
\[
S_i:=\{s\in S: \delta_i(s)\neq \infty \wedge \delta_{i+1}(s)=\infty \},  
\]
form a definable partition of $S$, which induces a partition of $X$ as well (some of the $S_i$ might be empty). By restricting to $S=S_m$ for some $m\in \{1,\ldots,{{N}\choose{k}} \}$, we find that $A_{s,\delta_1(s)},\ldots,A_{s,\delta_{m}(s)}$ are all the multi-balls that appear as fibers of $A_{s}$ over $X_s$. To conclude the proof we use the functions $\delta_1,\ldots,\delta_{m}$ to partition $X$ as follows. For $1\leqslant i\leqslant m$ let  
\[
X_i:=\{(s,\gamma)\in X: A_{s,\gamma}=A_{s,\delta_i(s)}\}. 
\]
We clearly have that the sets $\{X_1,\ldots,X_m\}$ form a partition of $X$. Moreover, by construction, for every $s\in S$ the fibers of $A_s$ over $(X_i)_s$ are constant and equal to $A_{s,\delta_i(s)}$. 
\end{proof}

\section{Auxiliary results on the form of the elements of $\cC_{\textnormal{exp}}^*(X)$}\label{section: auxiliary 2}

In the following lemma we state an elementary yet important result on integration of the character $\psi$ over balls of valuation radius at most 0.

{\begin{lemma}\label{lemma:cancel}
Let $a \in K$ and $\gamma \in \ZZ$ with $\gamma \leqslant 0$, then
\[
\int_{B_\gamma(a)} \psi(x) |dx| = 0.
\]
\end{lemma}

\begin{proof}
Since $\psi_{|\cO_K} \neq 1$, there exists $g \in \cO_K$ such that $\psi(g) \neq 1$. The ball $B_\gamma(a)$ is a disjoint union of $q_K^{-\gamma+1}$ balls from $\pball$ and it is easy to see that the map $x \mapsto x+g$ permutes these balls. So if we take $R$ to be a set of representatives from each of those balls, then
\[
B_{\gamma}(a) = \bigcup_{b\in R} B_1(b) = \bigcup_{b \in R} B_{1}(b+g).
\]
The character $\psi$ is constant on each of the balls $B_1(b)$, which have volume $q_K^{-1}$, so 
\begin{equation}\label{eq:permutation}
\int_{B_{\gamma}(a)} \psi(x) |dx| = q_K^{-1} \sum_{b \in R} \psi(b) = q_K^{-1} \sum_{b \in R} \psi(b + g) = q_K^{-1} \psi(g) \sum_{b \in R} \psi(b).
\end{equation}
Since $\psi(g) \neq 1$, \eqref{eq:permutation} can only hold if $\sum_{b \in R} \psi(b)=0$, which implies that $\int_{B_{\gamma}(a)} \psi(x)|dx| = 0$.
\end{proof}
Recall that in the Definition \ref{def:exp*functions}, the character $\psi$ is summed over multi-balls that consist purely of \textbf{maximal} balls from $\pball$.
The above lemma explains why it makes sense to impose that restriction.
\\\\
The following lemma gives a useful description of the form of exponential*-constructible functions.

\begin{lemma}\label{lem:general-form}
Let $X$ be a definable set. Then $\cC_{\exp}^*(X) = W(X)$, where
\[
W(X) := \left\{\sum_{i\in I} h_i(x) \sum_{B\sqsubseteq A^{i}_x} \psi(B) \Bigm\vert \begin{array}{l}I \text{ finite set, } h_i \in \cC(X), A^{i} \subseteq X\times K\\ \text{definable multi-ball of order } k_i \text{ on }\pball\end{array}\right\}.
\]
\end{lemma}

\begin{proof}
The inclusion $W(X) \subseteq \cC_{\exp}^*(X)$ is clear from the definition of $\cC_{\exp}^*(X)$. Furthermore,  $W(X)$ contains the generators of $\cC_{\exp}^*(X)$ and  is closed under addition and scalar multiplication by elements of $\QQ$. Hence, it remains to show that $W(X)$ is closed under multiplication. Now consider
\[
\Big(\sum_{i \in  I} h_i(x) \sum_{B \sqsubseteq A^{i}_x} \psi(B)\Big) \cdot \Big(\sum_{j\in J} \tilde{h}_j(x) \sum_{\tilde{B} \sqsubseteq \tilde{A}^{j}_x} \psi(\tilde{B}) \Big) = \sum_{(i,j) \in I \times J} h_i(x) \tilde{h}_j(x) \Big( \sum_{\substack{B \sqsubseteq A^{i}_x, \\ \tilde{B} \sqsubseteq \tilde{A}^{j}_x}} \psi(B + \tilde{B}) \Big).
\]
The reader can check that $B + \tilde{B}$ is again a ball in $\pball$. For each $i \in I$, $j \in J$ and $r \geqslant 1$, there exist definable sets
\begin{align*}
D^{(i,j,\geqslant r)} &:= \left\{(x,b) \in X \times K \Bigm\vert \begin{array}{l} \exists b_1, \ldots, b_r \in A^{i}_x \ \exists \tilde{b}_1, \ldots, \tilde{b}_r \in \tilde{A}^{j}_x :\bigwedge_{k=1}^r (b = b_k + \tilde{b}_k)\ \wedge\\ \bigwedge_{k,l = 1}^r \big(\ord_p(b_k - b_l) \leqslant 0\ \wedge \ord_p(\tilde{b}_k - \tilde{b}_l)\leqslant 0\big) \end{array}\right\};\\
D^{(i,j,r)} & := D^{(i,j, \geqslant r)} \backslash D^{(i,j,\geqslant r+1)};\\
 E^{(i,j,r)} &:= \left\{(x,b) \in D^{(i,j,r)} \mid
\exists b' \in K: \ord_p(b-b') \geqslant 0\ \wedge
b' \not \in D^{(i,j,r)}_x 
\right\}.
\end{align*}
Each fiber $D^{(i,j,r)}_x$ consists of the balls from $\pball$ that can be written in exactly $r$ ways as the sum of a ball $B \sqsubseteq A^{i}_x$ and a ball $\tilde{B} \sqsubseteq \tilde{A}^{j}_x$. Now $E^{(i,j,r)}_x$ contains only those balls in $D^{(i,j,r)}_x$, that are maximal in $D^{(i,j,r)}_x$. Remark that for $r > k_{(i,j)}:=\min\{k_i, \tilde{k}_j\}$ (where $k_i$ is the order of $A^{i}$ and $\tilde{k}_j$ is the order of $\tilde{A}^{j}$), $D_x^{(i,j,r)}$ and $E_x^{(i,j,r)}$ are empty, so we can write
\begin{align*}
\sum_{(i,j) \in I \times J} h_i(x) \tilde{h}_j(x) \Big( \sum_{\substack{B \sqsubseteq A^{i}_x, \\ \tilde{B} \sqsubseteq \tilde{A}^{j}_x}} \psi(B + \tilde{B}) \Big) &= \sum_{(i,j) \in I \times J} \sum_{r=1}^{k_{(i,j)}} r h_i(x) \tilde{h}_j(x) \Big( \sum_{\substack{B \subseteq D^{(i,j,r)}_x, \\ B \in \pball}} \psi(B)\Big) \\
&=\sum_{(i,j) \in I \times J} \sum_{r=1}^{k_{(i,j)}} r h_i(x) \tilde{h}_j(x) \Big( \sum_{B \sqsubseteq E^{(i,j,r)}_x} \psi(B)\Big).
\end{align*}
Here we have used the fact that $\int_{D_x^{(i,j,r)} \backslash E_x^{(i,j,r)}} \psi(b) |db| = 0$ which is a consequence of Lemma \ref{lemma:cancel}.

Unfortunately, the set $E^{(i,j,r)}$ is not necessarily a multi-ball, because different fibers might contain a different number of maximal balls. However, we do know that for $1 \leqslant r \leqslant k_{(i,j)}$, each fiber $E^{(i,j,r)}_x$ contains at most $k_i\tilde{k}_j$ maximal balls from $\pball$. Hence we can partition the set $X$ into definable sets $X^{(i,j,r)}_t$, for $1 \leqslant t \leqslant k_i \tilde{k}_j$, such that $E^{(i,j,r)}_x$ contains exactly $t$ maximal balls, for each $x \in X^{(i,j,r)}_t$. Remark that for each $E^{(i,j,r)}$ we might have to consider a different partition of $X$.

Now, for each $1 \leqslant t \leqslant k_i\tilde{k}_j$, we fix a set $H_t \subseteq K$ of $t$ maximal balls from $\pball$. Then we define (with parameters) a subset of $X \times K$, of which each fiber consists of $t$ maximal balls from $\pball$:
\[
E^{(i,j,r,t)} := \{(x,b) \in E^{(i,j,r)} \mid x \in X^{(i,j,r)}_t\} \cup (X\backslash X^{(i,j,r)}_t) \times H_t .
\]
Then we find
\begin{align*}
&\Big(\sum_{i \in  I}h_i(x) \sum_{B \sqsubseteq A^{i}_x} \psi(B)\Big) \cdot \Big(\sum_{j\in J} \tilde{h}_j(x) \sum_{\tilde{B} \sqsubseteq \tilde{A}^{j}_x} \psi(\tilde{B}) \Big) =\\
&\sum_{(i,j) \in I \times J} \sum_{r=1}^{k_{(i,j)}} \sum_{t=1}^{k_i\tilde{k}_j} r h_i(x)\tilde{h}_j(x)\mathbbm{1}_{X^{(i,j,r)}_t}(x) \sum_{B \sqsubseteq E^{(i,j,r,t)}_x} \psi(B),
\end{align*}
where $\mathbbm{1}_{X^{(i,j,r)}_t} \colon X \to \ZZ$ denotes the characteristic function of the set $X^{(i,j,r)}_t$. The last expression is clearly an element of $W(X)$, since $r h_i\tilde{h}_j\mathbbm{1}_{X^{(i,j,r)}_t} \in \cC(X)$ and $E^{(i,j,r,t)}$ is a definable multi-ball of order $t$ on $\pball$.
\end{proof}

\begin{remark}
There are of course several ways of writing an $f \in \cC^*_{\exp}(X)$ as an element of $W(X)$. The main difficulty in proving Conjecture \ref{conj:normalform} lies in showing that for an integrable function $f$, at least one of those ways uses only constructible functions $h_i$ that are \textbf{integrable} themselves.
\end{remark}


\section{Closedness under integration of $\ZZ$-variables}\label{section: Z}

Let us now prove Theorem \ref{thm:Zvar}, which we restate here for the reader's convenience. 

\begin{customthm}{A} Let $X\subseteq S\times \ZZ^n$ be a definable set. Let $f\in \cC_{\exp}^*(X)$ be such that $\textnormal{Int}(f,\ZZ^n) = S$. Then there exists $g\in \cC_{\exp}^*(S)$ such that for all $s \in S$,
\[
g(s) = \int_{X_s} f(s,\gamma)  |d\gamma|.  
 \]
\end{customthm}

\begin{proof} First of all, notice that by Fubini it suffices to show the result for $n=1$. By Lemma \ref{lem:general-form}, we may suppose that $f$ has the following form 
\begin{equation}\label{eq:sum}
f(s,\gamma) = \sum_{i=1}^m c_iq_K^{\alpha_i(s,\gamma)}\prod_{k=1}^{r_i} \beta_{ik}(s,\gamma)\sum_{B\sqsubseteq A_{s,\gamma}^i} \psi(B),   
\end{equation}
where the $c_i$ are non-zero rational constants, $\alpha_i, \beta_{ik}$ are definable functions from $X$ to $\Z$ and $A^i$ are definable multi-balls on $\pball$ over $X$. Let $I$ denote the set $I:=\{1,\ldots,m\}$. 

\

By iterating Proposition \ref{prop:Gamma-multiball}, there is a finite definable partition of $X$ into sets $\{X_w\}_{w\in W}$ such that for each $s\in S_w:=\pi_S(X_w)$, and for each $i\in I$, the multi-ball $A^i_s$ has constant fibers over $(X_w)_s$. Without loss of generality, suppose from now on that $X$ is one such piece $X_w$. Therefore, for all $s\in S$ the function
\[
e_i: (s,\gamma)\mapsto \sum_{B\sqsubseteq A_{s,\gamma}^i} \psi(B) 
\]
does not depend on $\gamma$. Notice that this function is an element of $\cC_{\exp}^*(S)$. Indeed, the set $E^i\subseteq S\times K$ defined by having fibers $E^i_s:=\bigcup_{\gamma\in X_s} A_{s,\gamma}^i$, is a multi-ball on $\pball$ over $S$, which shows that the function $e_i\colon s\mapsto \sum_{B\sqsubseteq E^i_{s}} \psi(B)$ is in $\cC_{\exp}^*(S)$. By multiplying $e_i$ by the constant $c_i$, we may omit such constants and rewrite equation (\ref{eq:sum}) as 
\begin{equation}\label{eq:sum2}
f(s,\gamma) = \sum_{i\in I} e_i(s)q_K^{\alpha_i(s,\gamma)}\prod_{k=1}^{r_i} \beta_{ik}(s,\gamma).   
\end{equation}

By Theorem \ref{thm:recti} we may further suppose that $X$ is a $\Gamma$-cell of the form 
\begin{equation}\label{cell1}
X = \left\{ (s,\gamma) \in S \times \ZZ \left| \begin{array}{l} \theta_1(s)  \ \square_1 \ \gamma \ \square_2 \ \theta_2(s),\\ \gamma \equiv l \mod M \end{array}\right\}\right.,
\end{equation}
%
for $\theta_1, \theta_2$ definable functions from $S$ to $\Z$ and $l, M\in \NN$ with $M>0$, and that for all $s\in S$, the functions $\alpha_i(s,\cdot)$ and $\beta_{ik}(s,\cdot)$ are linear in $\frac{\gamma-l}{M}$. From now on, we will denote $\frac{\gamma-l}{M}$ by $\zeta$.  Writing products of linear terms as polynomials in $\zeta$, we have that, for each $i\in I$,
\begin{equation}\label{hi}
q_K^{\alpha_i(s,\gamma)}\prod_{k=1}^{r_i} \beta_{ik}(s,\gamma)= q_K^{a_i\zeta+\delta_i(s)}\sum_{k=0}^{r_i} d_{ik}(s)\zeta^k, 
\end{equation}
where $a_i\in \ZZ$, $\delta_i$ is a definable function from $S$ to $\ZZ$ and $d_{ik}\in \cC(S)$. Since $q_K^{\delta_i(s)}\in \cC_{\exp}^*(S)$, we may assume that $\delta_i(s)=0$ by merging this factor into the functions $e_i(s)$. Therefore we have that

\begin{equation}\label{eq:sum3}
f(s,\gamma) = \sum_{i\in I} e_i(s) q_K^{a_i\zeta}\sum_{k=0}^{r_i} d_{ik}(s)\zeta^k.    
\end{equation}

After merging terms with the same factor $q_K^{a_i\zeta}$, we may suppose that $a_i\neq a_j$ for all $i,j\in I$ such that $i\neq j$. Set
\[
h_i(s,\gamma):=q_K^{a_i\zeta}\left(\sum_{k=0}^{r_i} d_{ik}(s)\zeta^k\right).
\]
\begin{claim}\label{claim:integrable} If for each $s\in S$ and $i\in I$ the functions $h_{i}(s,\cdot)$ are integrable over $X_s$, then there is $g\in\cC_{\exp}^*(S)$ such that $g(s)=\int_{X_s} f(s,\gamma) |d\gamma|$. 
\end{claim}

Notice that for each $i\in I$, $h_i$ is $\cL_2$-constructible. Therefore, by Theorem \ref{thrm:Cubides-Leenknegt}, letting $g_i(s)=\int_{X_s} h_{i}(s,\gamma)|d\gamma|$ we have that 
\begin{align*}
\int_{X_s} f(s,\gamma)|d\gamma | 	& = \int_{X_s} \left(\sum_{i\in I}e_i(s)h_{i}(s,\gamma)\right) |d\gamma |\\
										& = \sum_{i\in I} e_i(s)\int_{X_s} h_{i}(s,\gamma)|d\gamma | = \sum_{i\in I} e_i(s)g_i(s),
\end{align*}
which is a function in $\cC_{\exp}^*(S)$. This completes the proof of the claim. 

\

We will finish the argument by splitting in cases depending on the possible values of $\square_1$ and $\square_2$. 

\

\textbf{Case 1:} Suppose that $\square_1=\square_2=\text{`$<$'}$. In this case the set $X_s$ is finite for each $s \in S$, hence the functions $h_{i}(s,\cdot)$ are integrable over $X_s$. The result follows now by Claim \ref{claim:integrable}. 

\

\textbf{Case 2:} Suppose that $\square_1=\text{`$<$'}$ and $\square_2=\text{`$\emptyset$'}$. Let $j\in I$ be such that $a_{j}=\max_{i\in I} a_i$. In this case, if $a_j\geqslant 0$, then $e_j(s)\sum_{k=1}^{r_j}d_{jk}(s)\zeta^k=0$ for all $(s,\gamma)\in X$, since $f(s, \cdot)$ must be integrable over $X_s$ for each $s \in S$. If this holds, then 
\[
f(s,\gamma)=\sum_{i\neq j, i\in I} e_i(s) q_K^{a_i\zeta}\left(\sum_{k=0}^{r_i} d_{ik}(s)\zeta^k\right).
\]
By induction on $|I|$, either $f$ is identically $0$ or we may suppose that $a_j<0$. In the first case, the theorem clearly holds. In the second case, we have that each function $h_i(s, \cdot)$ is integrable over $X_s$ and we conclude again by Claim \ref{claim:integrable}. A similar argument handles the case $\square_2=\text{`$<$'}$ and $\square_1=\text{`$\emptyset$'}$. Finally, the case $\square_1=\square_2=\text{`$\emptyset$'}$ also reduces to this case by partitioning $X$ into $X_1$ and $X_2$ where 
\[
X_1:=\{(s,\gamma)\in X: 0\leqslant \gamma\} \text{ and } X_2:=\{(s,\gamma)\in X: \gamma<0\}. \qedhere
\]
 
\end{proof}


\section{Closedness under integration of $K$-variables}\label{section: K}

This section is dedicated to the proof of Theorem \ref{thm:Qpvar}. In the first subsection we deal with a special instance of the theorem. 

\subsection{Integrating characters over a definable subset of $K$}\label{section: integrate character}

The goal of this section is to show that functions of the form $s \mapsto \int_{X_s} \psi(x) |dx|$, where $X$ is a definable set, are always exponential*-constructible.  Note that this constitutes a  generalization of our observations in Example \ref{ex: motivation exp*}.

\begin{proposition}\label{prop: int over def set}
Let $X \subseteq S \times K$ be a definable set with bounded fibers $X_s$, i.e., for each $s \in S$, there exists $M \in \Z$ such that $\ord(x) \geqslant M$ for all $x \in X_s $. Then the function
\begin{equation}\label{int over def set}
s \mapsto \int_{X_s} \psi(x) |dx|
\end{equation}
is an element of $\cC^*_{\exp}(S)$.
\end{proposition}

For the proof of this proposition we will use  the clustered cell decomposition from Theorem \ref{thm:symmetrictrees}. This theorem states that one can partition the set $X$ into a finite union of classical cells and regular clustered cells. 
%
Recall that a regular clustered cell $C^\Sigma$ of order $k$ can be written as a disjoint union $C^{\sigma_1} \cup \ldots \cup C^{\sigma_k}$ of $k$ sets $C^{\sigma_i}$ that can (non-definably) be described as
\[ C^{\sigma_i} = \{(s,t) \in S \times K \mid \alpha(s) \ < \ \ord(t-\sigma_i(s)) \ < \ \beta(s) \wedge
 t-\sigma_i(s) \in \lambda Q_{n,m}\}.\]
%
%
%
Let us take a moment to explore the case where $X = C^{\Sigma}$.
By Lemma \ref{lemma:cancel} we may assume that the leaves of any fiber $X_s$ are balls of valuation radius at least 1, since leaves with smaller valuation radius will contribute nothing to the integral $\int_{X_s} \psi(x) |dx|$. Hence, we may assume that $\alpha(s) \geq -m$ for all $s \in S$. Note that this assumption will not affect the tree type of $C^{\Sigma}$. Furthermore, it implies 
that there exists a uniform bound on the number of balls in $\pball$ that have  non-empty intersection with $X_s$. For fixed $s$, the union of all these balls is equal to the corresponding fiber of the definable set
\[
\cB := \{(s,x) \in S \times K\mid \exists y \in X_s: \text{ord}(x-y) \geqslant 1\}.
\]

In the proof of Proposition \ref{prop: int over def set} we will partition this set into a finite number of definable multi-balls $\cB^1, \ldots, \cB^{i_0}$ of orders $k_1, \ldots, k_{i_0}$ on $\pball$, such that
\begin{equation}\label{eq right form}
\int_{X_s} \psi(x) |dx| = \sum_{j=1}^{i_0}g_j(s) \sum_{B \sqsubseteq \cB^j_{s}} \psi(B),
\end{equation}
for some $g_1, \ldots, g_{i_0} \in \cC(S)$. These constructible functions $g_i$ denote the volumes of the fibers of certain definable subsets of $S \times K$. 

\begin{proof}[Proof of Proposition \ref{prop: int over def set}:]

We will first consider the case where $X$ is a large, regular clustered cell $C^\Sigma$ of order $k$, using the notation from the previous discussion.  Recall that for such cells, all of the branching heights of $\Sigma_s$ occur below $\alpha(s)$, as discussed in Section \ref{cellsection}. 
As mentioned before, we may also assume that
 $\alpha(s) \geqslant -m$.

 

\

For each ball $B \subseteq \cB_s$ from $\pball$ we want to analyze the set $B \cap X_s$ and its volume. In most cases this set will consist of exactly one leaf from one of the sets $C^{\sigma_i(s)}$, but in some cases several leaves (possibly from different sets $C^{\sigma_{i'}(s)}$) could be contained within the ball $B$. We will have to distinguish between these two cases. Let $\sigma_i$ be a (non-definable) section of $\Sigma$ and $C^{\sigma_i(s), \gamma}$ the leaf of $C^{\sigma_i(s)}$ at height $\gamma$. This leaf has volume $q_K^{-(\gamma +m)}$, which is at most $q_K^{-1}$, since $-(\gamma+m) \leqslant -\gamma + \alpha(s) \leqslant -1$. Depending on the height $\gamma$, two cases may occur.
\begin{figure}[h]
\begin{minipage}{0.4 \linewidth}
\begin{tikzpicture}[scale=0.7]


\draw [dotted] (3,5.5) -- (3,5.8);

\path[fill = gray] (2,3) -- (1,3.5) -- (1.5,4) -- cycle;
\path[fill = gray] (-0.5,3) -- (-1.5,3.5) -- (-1,4) -- cycle;
\path[fill = gray] (4.5,3) -- (3.5,3.5) -- (4,4) -- cycle;

\draw[gray] (2,3) -- (1,3.5) -- (1.5,4) -- cycle;
\draw[gray] (-0.5,3) -- (-1.5,3.5) -- (-1,4) -- cycle;
\draw[gray] (4.5,3) -- (3.5,3.5) -- (4,4) -- cycle;

\draw (3,-1.5) -- (3,5.5);
\draw (3,4.25) -- (1.75,5.5);
\draw (3,3.5) -- (1.75,4.75);
\draw (3,2.75) -- (1.75,4);
\draw (3,2) -- (1.75,3.25);


\draw (1.75,5.5) -- (1.1,5.85) -- (1.4,6.15) -- cycle;
\draw (1.75,4.75) -- (1.1,5.1) -- (1.4,5.4) -- cycle;
\draw (1.75,4) -- (1.1,4.35) -- (1.4,4.65) -- cycle;\draw (1.75,3.25) -- (1.1,3.6) -- (1.4,3.9) -- cycle;



\draw [dotted] (0.5,5.5) -- (0.5,5.8);

\draw (0.5,0) -- (0.5,5.5);
\draw (0.5,4.25) -- (-0.75,5.5);
\draw (0.5,3.5) -- (-0.75,4.75);
\draw (0.5,2.75) -- (-0.75,4);
\draw (0.5,2) -- (-0.75,3.25);


\draw (-0.75,5.5) -- (-1.4,5.85) -- (-1.1,6.15) -- cycle;
\draw (-0.75,4.75) -- (-1.4,5.1) -- (-1.1,5.4) -- cycle;
\draw (-0.75,4) -- (-1.4,4.35) -- (-1.1,4.65) -- cycle;
\draw (-0.75,3.25) -- (-1.4,3.6) -- (-1.1,3.9) -- cycle;



\draw [dotted] (5.5,5.5) -- (5.5,5.8);

\draw (5.5,0) -- (5.5,5.5);
\draw (5.5,4.25) -- (4.25,5.5);
\draw (5.5,3.5) -- (4.25,4.75);
\draw (5.5,2.75) -- (4.25, 4);
\draw (5.5,2) -- (4.25,3.25);


\draw (4.25,5.5) -- (3.6,5.85) -- (3.9,6.15) -- cycle;
\draw (4.25,4.75) -- (3.6,5.1) -- (3.9,5.4) -- cycle;
\draw (4.25,4) -- (3.6,4.35) -- (3.9,4.65) -- cycle;
\draw (4.25,3.25) -- (3.6,3.6) -- (3.9,3.9) -- cycle;


\draw [dashed] (-1.7,3) -- (5.9,3);
\node[{right}] at (5.9,3){\footnotesize 1};
\draw [dotted] (-1.7,2) -- (5.9,2);
\node[{right}] at (5.9,2){\footnotesize $\gamma$};
\draw [dashed] (-1.7,1.5) -- (5.9,1.5);
\node[{right}] at (5.9,1.5){\footnotesize $-m$};

\draw (0.5,0) -- (3,-1);
\draw (5.5,0) -- (3,-1);
\end{tikzpicture}
 \caption{}
\end{minipage}
\begin{minipage}{0.4 \linewidth}
\begin{tikzpicture}[scale=0.7]


\draw [dotted] (3,5.5) -- (3,5.8);

\path[fill=gray] (3,3) -- (3.1,6.1) arc(50:162:1.5) -- cycle;
\path[fill=gray]  (0.5,3) -- (0.6,6.1) arc(50:162:1.5) -- cycle;
\path[fill=gray]  (5.5,3) -- (5.6,6.1) arc(50:162:1.5) -- cycle;

\draw[gray] (3,3) -- (3.1,6.1) arc(50:162:1.5) -- cycle;
\draw[gray] (0.5,3) -- (0.6,6.1) arc(50:162:1.5) -- cycle;
\draw[gray] (5.5,3) -- (5.6,6.1) arc(50:162:1.5) -- cycle;

\draw [dotted] (3,5.5) -- (3,5.8);

\draw (3,-1.5) -- (3,5.5);
\draw (3,4.25) -- (1.75,5.5);
\draw (3,3.5) -- (1.75,4.75);
\draw (3,2.75) -- (1.75,4);
\draw (3,2) -- (1.75,3.25);


\draw (1.75,5.5) -- (1.1,5.85) -- (1.4,6.15) -- cycle;
\draw (1.75,4.75) -- (1.1,5.1) -- (1.4,5.4) -- cycle;
\draw (1.75,4) -- (1.1,4.35) -- (1.4,4.65) -- cycle;
\draw (1.75,3.25) -- (1.1,3.6) -- (1.4,3.9) -- cycle;



\draw [dotted] (0.5,5.5) -- (0.5,5.8);

\draw (0.5,0) -- (0.5,5.5);
\draw (0.5,4.25) -- (-0.75,5.5);
\draw (0.5,3.5) -- (-0.75,4.75);
\draw (0.5,2.75) -- (-0.75,4);
\draw (0.5,2) -- (-0.75,3.25);


\draw (-0.75,5.5) -- (-1.4,5.85) -- (-1.1,6.15) -- cycle;
\draw (-0.75,4.75) -- (-1.4,5.1) -- (-1.1,5.4) -- cycle;
\draw (-0.75,4) -- (-1.4,4.35) -- (-1.1,4.65) -- cycle;
\draw (-0.75,3.25) -- (-1.4,3.6) -- (-1.1,3.9) -- cycle;



\draw [dotted] (5.5,5.5) -- (5.5,5.8);

\draw (5.5,0) -- (5.5,5.5);
\draw (5.5,4.25) -- (4.25,5.5);
\draw (5.5,3.5) -- (4.25,4.75);
\draw (5.5,2.75) -- (4.25, 4);
\draw (5.5,2) -- (4.25,3.25);


\draw (4.25,5.5) -- (3.6,5.85) -- (3.9,6.15) -- cycle;
\draw (4.25,4.75) -- (3.6,5.1) -- (3.9,5.4) -- cycle;
\draw (4.25,4) -- (3.6,4.35) -- (3.9,4.65) -- cycle;
\draw (4.25,3.25) -- (3.6,3.6) -- (3.9,3.9) -- cycle;


\draw [dotted] (-1.7,4.25) -- (5.9,4.25);
\node[{right}] at (5.9,4.25){\footnotesize $\gamma$};
\draw [dashed] (-1.7,3) -- (5.9,3);
\node[{right}] at (5.9,3){\footnotesize 1};
\draw [dashed] (-1.7,1.5) -- (5.9,1.5);
\node[{right}] at (5.9,1.5){\footnotesize $-m$};

\draw (0.5,0) -- (3,-1);
\draw (5.5,0) -- (3,-1);
\end{tikzpicture}
\caption{}
\end{minipage}
\end{figure}
\begin{itemize}
\item[(1)] If $\alpha(s) < \gamma \leqslant 0$, then the unique ball $B \subseteq \cB_s$ from $\pball$ which contains $C^{\sigma_i(s), \gamma}$, contains no other leaves of $C^{\sigma_i(s)}$. Since all the branching heights of $\Sigma_s$ occur below $\alpha(s)$, hence below $0$, $B$ does not intersect any of the other $k-1$ sets $C^{\sigma_{i'}(s)}$. This is the situation depicted in Figure 1.
\item[(2)] If $\gamma >0$, then $C^{\sigma_i(s), \gamma}$ is contained in the ball $B_1(\sigma_i(s)) \subseteq \cB_s$ from $\pball$. For any other (non-definable) section $\zeta$ for which $\sigma_i(s)$ and $\zeta(s)$ are equivalent, we know that $\text{ord}(\sigma_i(s) - \zeta(s)) > \alpha(s) + m \geqslant 0$, hence $B_1(\sigma_i(s)) = B_1(\zeta(s))$. Thus $\cB_s$ contains at most $k$ of these balls. This is the situation depicted in Figure 2.
\end{itemize}
As we have already mentioned, we want to partition the set $\cB$ in a definable way. For each $\gamma$ of  type (1), we will define a set consisting of all the balls that contain a leaf at height $\gamma$. The balls that contain a leaf of type (2) will be collected in an additional definable set. We will now explain how to define these sets uniformly in $s$. For this, note that the leaves of $X_s$ are also the maximal balls of $X_s$, since all the branching heights of $\Sigma_s$ occur below $\alpha(s)$.
\\\\
%
We inductively define, for each $j \geqslant 1$, a definable set $X^{(j)}$ and a definable function $d_j: S \to \Z \cup \{\infty\}$ as follows. For each $s \in S$,   the set $X^{(j)}_s$ is the set containing the largest leaves in $X_s\backslash \left(\cup_{l=1}^{j-1}X^{(l)}_s\right)$. The function $d_j$ is such that for each $s \in S$, the volume of the leaves in $X^{(j)}_s$  equals $q_K^{-d_j(s)}$ when the set $X^{(j)}_s$ is not empty, and $d_j(s) = \infty$ otherwise. 
Note that one always has $d_j(s) \geqslant 1$. Furthermore, $d_j(s) \leqslant m$ if and only if the leaves in $X^{(j)}_s$ are leaves of type (1). 
\\\\
Now let $i_0(s)$ be the smallest positive integer for which $d_{i_0(s)}(s) > m$. This integer could depend on $S$, but in any case, $i_0(s)$ is uniformly bounded on $S$, by $m+1$. Thus there exists a finite definable partition of $S$, such that $i_0(s)$ is constant on each of the sets in the partition. By restricting our clustered cell to any set in this partition,
we may assume that $i_0$ is constant on $S$. This means that for all $s \in S$ there are $i_0-1$ definable sets of leaves of type (1), $X^{(1)}_s,\ldots, X^{(i_0-1)}_s$. Note that each of these sets contains exactly $k$ leaves, one for each equivalence class of centers in $\Sigma_s$. 
\\\\
For each $1 \leqslant j \leqslant i_0-1$, define
\[
\cB^j := \{(s,x) \in S \times K \mid \exists y \in X^{(j)}_s: \text{ord}(x-y) \geqslant 1\}
\]
to be the definable set whose fibers $\cB^j_s$ contain all balls in $\pball$ that have a nonempty intersection with $X^{(j)}_s$. With this definition, the sets $\cB^j$ are disjoint subsets of $\cB$ and each of them is a multi-ball of order $k$ on $\pball$, since all the branching heights of $\Sigma_s$ occur below $\alpha(s)$.
For each $1 \leqslant j \leqslant i_0-1$, we now have that
\begin{equation}\label{subsum fix vol}
\int_{X^{(j)}_s} \psi(x) |dx| = \sum_{B \sqsubseteq \cB^j_s} \psi(B) \cdot \text{Vol}(B \cap X_s) = q_K^{-d_j(s)} \cdot \sum_{B \sqsubseteq \cB^j_s} \psi(B) \in \cC^*_{\text{exp}}(S), 
\end{equation}
where we use the fact that $B\cap X_s$ is one leaf of $X_s$ with volume $q_K^{-d_j(s)}$.

\

For all $j \geqslant i_0$ the leaves in $X^{(j)}_s$ are leaves of type (2). The union of these leaves is the definable set
\[
X'_s := X_s \backslash \left(\cup_{l=1}^{i_0-1}X^{(l)}_s\right).
\]
The balls from $\pball$ in $\cB_s$ that have nonempty intersection with $X'_s$ make up the fibers of the definable set
\begin{align*}
\cB^{i_0} &:= \{(s,x) \in S \times K \mid \exists y \in X'_s: \text{ord}(x-y) \geqslant 1\}\\ 
&= \cB \backslash \left(\cup_{l=1}^{i_0-1} \cB^{l}\right).
\end{align*}
Note that if $\Sigma_s$ has branching heights above $1$, then even for certain $\sigma_i, \sigma_{i'}$ not equivalent at $s$, we will have $B_1(\sigma_i(s)) = B_1(\sigma_{i'}(s))$. Therefore it could happen that $\cB_s$ contains strictly less than $k$ of these balls. Let us denote the number of balls in $\cB^{i_0}_s$ by $k_0(s)$. Since this number may change with $s$, the set $\cB^{i_0}$ is not necessarily a multi-ball. To make it into one, we will have to partition $S$, using the procedure described below.

 It could happen that the balls from $\pball$ in $\cB_s^{i_0}$ are not maximal balls. This happens exactly if $\Sigma_s$ has a branching height at 0 with $q_K$ branches. Let $S' \subseteq S$ be the definable set of $s$ for which this happens, then for each $s \in S'$,
\[
\int_{X'_s} \psi(x) |dx| = \sum_{\substack{B \subseteq \cB_s^{i_0},\\ B \in \pball}} \text{Vol}(X'_s \cap B) \psi(B) =0,
\]
where we have used Lemma \ref{lemma:cancel} and the fact that the volume of $X'_s \cap B$ is the same for each $B \subseteq \cB^{i_0}_s$ with $B \in \pball$, by condition $(ii)$ of Theorem \ref{thm:symmetrictrees}. The constant function 0 is clearly an exponential*-constructible function on $S'$.

From now on we may assume without loss of generality that $S'= \emptyset$.
Since we have $k_0(s) \leq k$ for all $s \in S$, we can (definably) partition $S$ further and reduce to the case where for all $s \in S$, $\cB_s^{i_0}$ contains exactly $k_0$ maximal balls from $\pball$.

By condition $(ii)$ of Theorem \ref{thm:symmetrictrees}, the tree associated to $\Sigma_s$ is highly symmetric, and hence
the sets $X'_s \cap (B_1(c))$ all have the same volume, for each $c \in \Sigma_s$. There are $k_0$ of these sets, hence each of them has volume $\frac{1}{k_0} \cdot \text{Vol}(X'_s)$, which is a constructible function on $S$. We can conclude that
\begin{equation}\label{eq: int over top part, max balls}
\int_{X'_s} \psi(x) |dx| = \frac{1}{k_0} \cdot \text{Vol}(X'_s) \sum_{B \sqsubseteq \cB^{i_0}_s} \psi(B) \in \cC_{\exp}^*(S).
\end{equation}
The equations \eqref{subsum fix vol} and \eqref{eq: int over top part, max balls} give us the form of \eqref{eq right form}.

\

Now consider the case where $X$ is a small clustered cell. For such a cell, each fiber only has leaves at a single height, but we can no longer assure that the branching heights will necessarily occur below $\alpha(s)$. Still, the reader can check that this case can be proven similarly as the case of large cells, by partitioning $S$ in the same way as for case (2) above. The proof for classical cells follows the same structure as the proof for large cells and will be left to the reader. This concludes the proof of this theorem.
%
\end{proof}

\subsection{The proof of Theorem \ref{thm:Qpvar}}

We are ready to prove Theorem \ref{thm:Qpvar}, which we restate for the reader's convenience.

\begin{customthm}{B}\label{thm:Qpvar}  Let $X \subseteq S \times K^n$ be a definable set and $f\in \cC^*_\text{exp}(X)$. If $f$ can be written in $n$-normal form, then there exists $g\in \cC_{\exp}^*(S)$ such that, for all $s \in S$,
 \[
g(s) = \int_{X_s} f(s,x) |dx|.
 \]
\end{customthm}


\begin{proof}
 Since $f$ can be written in $n$-normal form, by additivity of integration we are reduced to prove the case where $f$ is of the form
\[
f(s,x) := h(s,x) \sum_{B \sqsubseteq A_{s,x}} \psi(B),
\]
with $A \subseteq X\times K$ a definable multi-ball on $\pball$, $\textnormal{Int}(h,K^n) = S$ and $\textnormal{Int}(f, K^n) = S$.

Put $Y := \{(s,b,x) \in S \times K^{n+1} \mid b \in A_{s,x}\}$, which can have empty fibers $Y_{s,b}$ for some $(s,b) \in S \times K$. Since $h(s, \cdot)$ is integrable over $X_s$ for each $s \in S$, we can apply Fubini to change the order of integration. Hence
\begin{align*}
\int_{X_s} h(s,x) \cdot \sum_{B\sqsubseteq A_{s,x}} \psi(B) |dx| &= q_K\int_{X_s} \int_{A_{s,x}}h(s,x) \psi(b) |db||dx|\\
&= q_K \int_{K} \int_{Y_{s,b}} h(s,x)\psi(b) |dx||db|\\
&= q_K \int_K \psi(b) \Big( \int_{Y_{s,b}} h(s,x) |dx| \Big) |db|.
\end{align*}
The set $Y_{s,b}$ is a definable subset of $K^n$, so by Theorem \ref{thrm:Cubides-Leenknegt} the function
\[
S \times K \to \QQ: (s, b) \mapsto \int_{Y_{s,b}} h(s,x) |dx|
\]
is a constructible function. This means that there exist definable functions $\alpha_i\colon  S \times K \to \ZZ$, $\beta_{ij}\colon  S \times K \to \ZZ$ and constants $c_i \in \QQ$ such that, for all $s \in S$, $b \in K$,
\[
\int_{Y_{s,b}} h(s,x) |dx| = \sum_{i=1}^m c_i q_K^{\alpha_i(s,b)} \prod_{j=1}^{r} \beta_{ij}(s,b).
\]
For each $\gamma = (\gamma_i)_i \in \ZZ^{m(r+1)}$ we denote the rational number $\sum_{i=1}^m c_i q_K^{\gamma_i} \prod_{j=1}^{r} \gamma_{jm + i}$ by $C_\gamma$ and we consider the definable set
\[
D := \{(s,\gamma,b)\in S \times \ZZ^{m(r+1)} \times K \mid \big(\alpha_1(s,b), \ldots, \alpha_m(s,b), \beta_{11}(s,b), \ldots, \beta_{mr}(s,b)\big) = \gamma\},
\]
Notice that for a fixed $s \in S$, the fibers $D_{s,\gamma}$ partition $K$. We can definably distinguish between the fibers $D_{s, \gamma}$ that are bounded and the ones that are not:
\[
G := \{(s,\gamma) \in S \times \ZZ^{m(r+1)} \mid \exists \delta \in \ZZ : D_{s,\gamma} \subseteq B_\delta(0)\}.
\]
 This gives us, for each $s \in S$,
\begin{align*}
g(s) &= \int_{X_s} h(s,x) \cdot \sum_{B\sqsubseteq A_{s,x}} \psi(B) |dx|\\
&= q_K \cdot \int_{K} \psi(b) \sum_{i=1}^m c_i q_K^{\alpha_i(s,b)} \prod_{j=1}^{r} \beta_{ij}(s,b) |db|\\
&= q_K \sum_{\gamma \in \ZZ^{m(r+1)}} \int_{D_{s,\gamma}} C_\gamma\psi(b) |db|\\
&=q_K \sum_{\gamma \in G_s} \int_{D_{s,\gamma}} C_\gamma \psi(b) |db| + q_K \sum_{\gamma \in \ZZ^{m(r+1)} \backslash G_s} \int_{D_{s,\gamma}}C_\gamma\psi(b) |db|\\
&=q_K \int_{G_s} \Big(C_\gamma \int_{D_{s,\gamma}} \psi(b) |db|\Big) |d\gamma|,
\end{align*}
where the last equality follows from the fact that each of the integrals $\int_{D{s,\gamma}} C_\gamma \psi(b) |db|$ exists, which means that if $\gamma \notin G_s$, then $C_\gamma = 0$. Using Proposition \ref{prop: int over def set} we see that the function
\[
l \colon (s, \gamma) \mapsto C_\gamma \int_{D_{s,\gamma}} \psi(b) |db|
\]
is an exponential*-constructible function on $G \subseteq S\times \ZZ^{m(r+1)}$. By applying Theorem \ref{thm:Zvar} to $l$ we can conclude that $g \colon s \mapsto q_K \int_{G_s} l(s,\gamma) |d\gamma|$ is an exponential*-constructible function on $S$.
\end{proof}

\subsection*{Acknowledgements} 
The authors would like to thank Raf Cluckers for stimulating conversations during the preparation of this paper. The first author was supported by the European Research Council under the European Community's Seventh Framework Programme (FP7/2007-2013) with ERC Grant Agreement nr. 615722 (MOTMELSUM).
The second author was supported by ERC grant agreements nr. 615722
(MOTMELSUM) and nr. 637027 (TOSSIBERG). During the realization of this project, the third author was a postdoctoral fellow of the Fund for Scientific Research - Flanders (Belgium) (F.W.O.).

\bibliographystyle{acm}
\bibliography{../Bibliografie}

\begin{thebibliography}{10}

\bibitem{cham-cubi-leen}
{\sc Chambille, S., Cubides~Kovacsics, P., and Leenknegt, E.}
\newblock Clustered cell decomposition in {$P$}-minimal structures.
\newblock {\em Annals of Pure and Applied Logic 168}, 11 (2017), 2050 -- 2086.

\bibitem{cham-cubi-leen-note}
{\sc Chambille, S., Cubides~Kovacsics, P., and Leenknegt, E.}
\newblock A note on clustered cells, 2017.
\newblock arXiv:1703.03731 [math.LO].

\bibitem{clu-2003}
{\sc Cluckers, R.}
\newblock Analytic $p$-adic cell decomposition and integrals.
\newblock {\em Transactions of the American Mathematical Society 356}, 4
  (2004), 1489--1499.

\bibitem{Clu-Gor-Hal-14}
{\sc Cluckers, R., Gordon, J., and Halupczok, I.}
\newblock Integrability of oscillatory functions on local fields: transfer
  principles.
\newblock {\em Duke Math. J. 163}, 8 (2014), 1549--1600.

\bibitem{clu-loe-08}
{\sc Cluckers, R., and Loeser, F.}
\newblock Constructible motivic functions and motivic integration.
\newblock {\em Invent. Math. 173}, 1 (2008), 23--121.

\bibitem{cubi-leen}
{\sc Cubides~Kovacsics, P., and Leenknegt, E.}
\newblock Integration and cell decomposition in {$P$}-minimal structures.
\newblock {\em The Journal of Symbolic Logic 81}, 3 (2016), 1124--1141.

\bibitem{Dar-Hal-2017}
{\sc Darni{\`e}re, L., and Halupczok, I.}
\newblock Cell decomposition and classification of definable sets in p-optimal
  fields.
\newblock {\em Journal of Symbolic Logic 82}, 1 (March 2017), 120--136.

\bibitem{denef-85}
{\sc Denef, J.}
\newblock On the evaluation of certain {$p$}-adic integrals.
\newblock In {\em S{\'e}minaire de th{\'e}orie des nombres, {P}aris 1983--84},
  vol.~59 of {\em Progr. Math.} Birkh{\"a}user Boston, Boston, MA, 1985,
  pp.~25--47.

\bibitem{denef-86}
{\sc Denef, J.}
\newblock {$p$}-adic semi-algebraic sets and cell decomposition.
\newblock {\em J. Reine Angew. Math. 369\/} (1986), 154--166.

\bibitem{has-mac-97}
{\sc Haskell, D., and Macpherson, D.}
\newblock A version of o-minimality for the {$p$}-adics.
\newblock {\em J. Symbolic Logic 62}, 4 (1997), 1075--1092.

\bibitem{kui-lee-13}
{\sc Kuijpers, T., and Leenknegt, E.}
\newblock Differentiation in {$P$}-minimal structures and a {$p$}-adic local
  monotonicity theorem.
\newblock {\em J. Symb. Log. 79}, 4 (2014), 1133--1147.

\bibitem{mou-09}
{\sc Mourgues, M.-H.}
\newblock Cell decomposition for {$P$}-minimal fields.
\newblock {\em MLQ Math. Log. Q. 55}, 5 (2009), 487--492.

\end{thebibliography}

\end{document}